\documentclass[letterpaper,11pt]{article}
\usepackage[margin=1in]{geometry}
\usepackage[T1]{fontenc}
\usepackage{lmodern}
\usepackage{amsmath,amssymb,amsthm,mathtools,microtype}
\usepackage{enumitem,aliascnt,xcolor}
\usepackage{graphicx,subfig,float}
\usepackage[colorlinks=true,linkcolor=blue,citecolor=red,urlcolor=blue]{hyperref}
\usepackage[nameinlink,noabbrev]{cleveref}

\newtheorem{theorem}{Theorem}[section]
\newtheorem{lemma}{Lemma}[section]
\newtheorem{proposition}{Proposition}[section]

\theoremstyle{remark}
\newtheorem{remark}{Remark}[section]

\makeatletter
\let\c@lemma\c@theorem
\let\c@proposition\c@theorem
\let\c@corollary\c@theorem
\let\c@remark\c@theorem
\makeatother

\newcommand\numberthis{\addtocounter{equation}{1}\tag{\theequation}}

\numberwithin{equation}{section}
\newcommand{\R}{\mathbb R}
\newcommand{\C}{\mathbb C}
\newcommand{\reals}{\mathbb{R}}

\newcommand{\integers}{\mathbb{Z}}
\newcommand{\complex}{\mathbb{C}}
\newcommand{\Expect}{\mathbb{E}}

\newcommand{\Prob}{\mathbb{P}}

\newcommand{\pth}[1]{\left( #1 \right)}

\newcommand{\sth}[1]{\left\{ #1 \right\}}

\newcommand{\todistr}{{\xrightarrow{{\rm D}}}}

\newcommand{\iprod}[2]{\left \langle #1, #2 \right\rangle}

\newcommand{\Indc}{\mathbf{1}}

\newcommand{\E}{\Expect}
\newcommand{\N}{\mathcal N}

\newcommand{\Nzero}{\mathbb{Z}_+}
\newcommand{\eps}{\varepsilon}
\def\d{\mathrm{d}}

\newcommand{\stepa}[1]{\overset{\rm (a)}{#1}}
\newcommand{\stepb}[1]{\overset{\rm (b)}{#1}}
\newcommand{\stepc}[1]{\overset{\rm (c)}{#1}}
\newcommand{\stepd}[1]{\overset{\rm (d)}{#1}}

\newcommand{\calE}{{\mathcal{E}}}

\newcommand{\calH}{{\mathcal{H}}}

\newcommand{\calN}{{\mathcal{N}}}

\newcommand{\calP}{{\mathcal{P}}}

\newcommand{\calS}{{\mathcal{S}}}

\newcommand{\calX}{{\mathcal{X}}}

\newcommand{\K}{\mathbf K}

\newcommand{\cP}{\mathcal P}

\newcommand{\norm}[1]{\left\lVert#1\right\rVert}

\newcommand{\pnorm}[2]{\lVert #1\rVert_{#2}}

\DeclareMathOperator{\supp}{supp}

\DeclareMathOperator{\Cov}{Cov}

\DeclareMathOperator{\Rea}{Re}
\DeclareMathOperator{\Ima}{Im}
\DeclareMathOperator{\diag}{Diag}

\allowdisplaybreaks[1]
\setlist[itemize]{topsep=4pt,itemsep=2pt}

\usepackage{prettyref}
\newrefformat{eq}{(\ref{#1})}
\newrefformat{e}{(\ref{#1})}
\newrefformat{thm}{Theorem~\ref{#1}}
\newrefformat{th}{Theorem~\ref{#1}}
\newrefformat{chap}{Chapter~\ref{#1}}
\newrefformat{sec}{Section~\ref{#1}}
\newrefformat{seca}{Section~\ref{#1}}
\newrefformat{algo}{Algorithm~\ref{#1}}
\newrefformat{fig}{Figure~\ref{#1}}
\newrefformat{tab}{Table~\ref{#1}}
\newrefformat{rmk}{Remark~\ref{#1}}
\newrefformat{clm}{Claim~\ref{#1}}
\newrefformat{def}{Definition~\ref{#1}}
\newrefformat{cor}{Corollary~\ref{#1}}
\newrefformat{lmm}{Lemma~\ref{#1}}
\newrefformat{lem}{Lemma~\ref{#1}}
\newrefformat{prop}{Proposition~\ref{#1}}
\newrefformat{pr}{Proposition~\ref{#1}}
\newrefformat{app}{Appendix~\ref{#1}}
\newrefformat{apx}{Appendix~\ref{#1}}
\newrefformat{ex}{Example~\ref{#1}}
\newrefformat{exer}{Exercise~\ref{#1}}
\newrefformat{soln}{Solution~\ref{#1}}

\begin{document}
\title{Self-regularization and structural adaptation of NPMLE for Gaussian mixture models in all dimensions}

\author{Yanjun Han\thanks{
Courant Institute of Mathematical Sciences and Center for Data Science, New York University, \texttt{yanjunhan@nyu.edu}.} , Yandi Shen\thanks{Department of Statistics and Data Science, Carnegie Mellon University, \texttt{yandis@andrew.cmu.edu}.} , and Yihong Wu\thanks{Department of Statistics and Data Science, Yale University, \texttt{yihong.wu@yale.edu}.}}
\date{\today}

\maketitle

\begin{abstract}

We establish structural properties of the nonparametric maximum likelihood estimator (NPMLE) for $d$-dimensional Gaussian mixtures whose means lie in a hypercube. 
We show that, with high probability, the NPMLE based on $n$ independent observations self-regularizes by fitting a mixture model with $O((\frac{\log n}{\log\log n})^d)$ components, achieving the optimal model complexity without explicit model selection. 
Furthermore, the NPMLE exhibits a structural adaptation property: For data drawn from a finite mixture, the number of fitted components is bounded in probability as $n\to\infty$. 

Our proof applies a multivariate Jensen-type bound to control the typical number of critical points of the KKT function associated with the
first-order optimality conditions. The key step is to lower bound its gradient norm uniformly on complex spheres. We establish this lower bound by decomposing the KKT function into an empirical fluctuation term and a relative density estimation error term, combining small-ball probability estimates for the former with polynomial approximation and Hellinger error bounds for the latter.
\end{abstract}

\tableofcontents

\section{Introduction}

The nonparametric maximum likelihood estimator (NPMLE) \cite{KieferWolfowitz1956} is a widely used method for fitting mixture models and learning priors in empirical Bayes estimation. Its appeal lies in the combination of statistical accuracy, computational efficiency, and flexibility (such as tuning-parameter freeness).
Consider a parametric family of densities $\{p_\theta:\theta\in\Theta\}$ with respect to a common dominating measure. Given observations $X_1,\ldots,X_n$, we maximize the total log-likelihood over the space 
$\cP(\Theta)$ of all probability measures on $\Theta$, resulting in the NPMLE:
\begin{equation}\label{eq:NPMLE-intro}
 \widehat G\in\operatorname*{arg\,max}_{G\in\cP(\Theta)}
 \sum_{i=1}^n\log p_G(X_i).
\end{equation}
Here 
\[
 p_G(x):=\int_\Theta p_\theta(x)\d G(\theta)
 \]
 is the mixture density and $G$ is referred to as the mixing distribution or prior.
 A principal example is the Gaussian location mixture model in $d$ dimensions, where $p_\theta$ is the density of $N(\theta,I)$ and $\theta\in\Theta\subseteq\R^d$. In this case, $p_G = G * N(0,I)$ is a Gaussian convolution. 

 Computing the NPMLE
 \prettyref{eq:NPMLE-intro} is an infinite-dimensional convex optimization problem and 
can be solved using ``on-the-grid'' methods by optimizing mixing weights over discretized $\Theta$  \cite{KoenkerMizera2014}, or 
``off-the-grid'' methods, including the Frank--Wolfe algorithm \cite{FrankWolfe1956} (known as the vertex direction method in the mixture literature \cite{Bohning1985}) and methods based on gradient flow \cite{yan2024learning}.

The study of the NPMLE has two complementary aspects:
\begin{enumerate}
\item \emph{Statistical properties:} Given $n$ independent observations from $p_{G_\star}$, whether the estimated mixing distribution $\widehat G$ and the fitted density $p_{\widehat G}$ approach the respective ground truth $G_{\star}$ and $p_{G_{\star}}$ as $n\to\infty$, and at what rates.
\item \emph{Structural properties:} Existence, uniqueness, and discreteness of the maximizer, 
in both worst case and average case.
Of particular interest is the cardinality of its support,  corresponding to the number of components fitted by the NPMLE.

\end{enumerate}

In one dimension, both aspects are relatively well-understood. 
Statistically, the NPMLE is known to attain (nearly) optimal rates for estimating the mixture density for a variety of prior classes \cite{GhosalVanDerVaart2001,Zhang2009,PolyanskiyWu2021,ChenWu2026,hundrieser2026minimax}.
Structurally,
the univariate Gaussian NPMLE is unique and discrete deterministically, with at most $n$ atoms \cite{Lindsay1983a,Lindsay1983b}. 
Furthermore, for iid observations with a subgaussian prior, 
\cite{PolyanskiyWu2020} showed that the NPMLE has $O(\log n)$ atoms with high probability. This phenomenon is termed \emph{self-regularization}: without explicitly penalizing model complexity, the NPMLE selects the right number of components, in the sense that $O(\log n)$ components are both necessary and sufficient for reaching the desired statistical accuracy.

In multiple dimensions, while the statistical theory for NPMLE has been extended in \cite{SahaGuntuboyina2020}, the structural property is far less understood and appears more nuanced. 
Unlike one dimension,  
\cite{SoloffGuntuboyinaSen2025} identified certain configuration of data points in two dimensions for which the NPMLE is not unique. 
Only recently, it was shown that 
for the unconstrained problem of $\Theta=\reals^d$, 
every NPMLE is finitely discrete \cite{Wang2026Finite} with $O(n^d)$ atoms deterministically, and that almost surely the solution is unique  with at most $n$ atoms \cite{Wang2026Unique}. These results settle qualitative structural questions, while leaving open the typical model complexity and its dependence on the true mixing distribution.

In this paper, we consider Gaussian mixture models in dimension $d\geq1$, with means lying in a bounded hypercube $\Theta=K$. The NPMLE $\widehat G$ is obtained by maximizing the likelihood \prettyref{eq:NPMLE-intro} over all distributions on $K$.
Our main contributions are as follows:
\begin{enumerate}
\item \emph{Self-regularization for arbitrary mixtures.} Uniformly over all true priors $G_\star\in\cP(K)$, the NPMLE satisfies
\begin{equation}
 |\supp(\widehat G)|
 =O_{\Prob}\left(\left(\frac{\log n}{\log\log n}\right)^d\right),
 \label{eq:main1-intro}
\end{equation}
matching the optimal complexity for approximating arbitrary mixtures up to statistically negligible error.

\item \emph{Structural adaptation for finite mixtures.} For any fixed finitely supported $G_\star$, we show that
\begin{equation}
|\supp(\widehat G)|=O_{\Prob}(1). 
    \label{eq:main2-intro}
\end{equation}
Thus the number of fitted components remains bounded in probability even though the likelihood is maximized over the infinitely overparametrized space of all priors.
\end{enumerate}
We note that both results are new even for $d=1$. In particular, \cite[Section~5.3]{PolyanskiyWu2020} raised the question whether the NPMLE over a compact interval has a stronger regularization effect, achieving model complexity $O_\Prob(\frac{\log n}{\log\log n})$ as opposed to  $O_\Prob(\log n)$ for unconstrained NPMLE. This is affirmed by \prettyref{eq:main1-intro}.

\subsection{Model and main theorems}
\label{sec:model}

Fix an integer $d\ge1$.  For a real vector $a=(a_1,\ldots,a_d)\in\R^d$, we write $\norm{a}=\sqrt{\sum_{j=1}^d a_j^2}$ for its Euclidean norm.
For $z\in\C^d$, write $\norm{z}^2=\norm{\Rea z}^2+\norm{\Ima z}^2$ for the squared Euclidean norm.

Fix an absolute constant $M>0$ and let $K=[-M,M]^d$. 
Denote the Gaussian density centered at $\theta\in\R^d$ by
\begin{equation}
 p_\theta(x)=\varphi(x-\theta)=(2\pi)^{-d/2}\exp\{-\norm{x-\theta}^2/2\},
 \qquad x\in\R^d,
 \label{eq:kernel}
\end{equation}
and for $G\in\cP(K)$ set
\begin{equation}
 p_G(x)=\int_K p_\theta(x)\d G(\theta).
 \label{eq:mixture}
\end{equation}
The observations $X_1,\ldots,X_n$ are iid from $p_\star\equiv p_{G_\star}$.  Let $P_n=n^{-1}\sum_{i=1}^n\delta_{X_i}$ be the empirical measure.  For an integrable function $f$, denote its population and empirical average as
\[
 \Expect_\star f\triangleq\int_{\R^d}f(x)p_\star(x)\d x,
 \qquad
 \Expect_n f\triangleq\int f\d P_n=\frac1n\sum_{i=1}^n f(X_i).
\]

Denote the NPMLE by
\begin{equation}
 \widehat G = \operatorname*{arg\,max}_{G\in\calP(K)} \ell_n(G), \quad \ell_n(G):=\sum_{i=1}^n\log p_G(X_i).
 \label{eq:likelihood}
\end{equation}
By \cite{Wang2026Unique}, this maximizer is unique almost surely\footnote{
The almost sure uniqueness in \cite{Wang2026Unique} is stated for unconstrained NPMLE but holds for the compactly supported version as well. Nevertheless, our theory does not rely on this uniqueness result, as all high-probability cardinality bounds apply to all NPMLE solutions simultaneously.} with respect to the sample $X_1,\ldots,X_n$.

Throughout the paper, for $n\geq 3$, write
\begin{equation}
 L_n:=\frac{\log n}{\log\log n}.
 \label{eq:Ln}
\end{equation}

\begin{theorem}[Self-regularization]
\label{thm:continuous-prior}
There are constants $B_0,C,c>0$, depending only on $K$, such that for every fixed $B\ge B_0$ and all sufficiently large $n$,
\begin{equation}\label{eq:tail-continuous-prior}
\sup_{G_\star\in\cP(K)}
\Prob_{G_\star}\Big(|\supp(\widehat G)|>B^{2d} L_n^d\Big)
\le Ce^{-cB\sqrt{L_n}}+n^{-c}.
\end{equation}
Consequently, uniformly over $G_\star\in\cP(K)$,
\begin{equation}\label{eq:OP-continuous-prior}
|\supp(\widehat G)|=O_{\Prob}(L_n^d).
\end{equation}
\end{theorem}

We note that the bound $O(L_n^d)$ cannot be improved in general, since this is the fewest number of components needed for the existing statistical guarantees of the NPMLE. Indeed, \cite[Theorem~2.1]{SahaGuntuboyina2020} (see also Lemma \ref{lem:universal-H}) shows that the squared Hellinger error
$H^2(p_{\widehat G},p_{G_\star}) :=
\int (\sqrt{p_{\widehat G}} - \sqrt{p_{G_\star}})^2$ is at most $ O_{\Prob}((\log n)^{d+1}/n)$, while \cite[Proposition 6]{MaWuYang2025} shows that there exists a mixing distribution $G_\star$ with density on $K$ such that 
the best approximation error of approximating $p_{G_\star}$ by finite Gaussian mixtures is at least
\begin{align*}
\inf_{G \in \calP_s(K)} H^2(p_{G_\star}, p_G) \geq \exp\Big(-Cs^{1/d}\log(s+1)\Big)
\end{align*}
for constants $C > 0$ that only depend on $K$, where $\calP_s(K)$ denotes the collection of $s$-atomic distributions supported on $K$. Thus for this $G_\star$ we must have $|\supp(\widehat G)| = \Omega_{\Prob}(L_n^d)$.
On the other hand, 
\cite[Proposition 5]{MaWuYang2025} shows that any mixture $p_{G_\star}$ with $G_\star\in \calP(K)$ can be approximated by a finite mixture with $O(L_n^d)$ components within total variation $o(n^{-1})$ so that the corresponding joint distributions on the $n$ observations are indistinguishable. In this sense, \prettyref{eq:tail-continuous-prior} gives the \textit{right} model complexity as there is little statistical benefit to fit a more complex mixture model.

Next, if the true prior $G_\star$ has a finite support, the following result states that NPMLE exhibits a strong adaptation property structurally; namely, its support size is bounded in probability.

\begin{theorem}[Structural adaptation]
\label{thm:main}
Let $G_\star\in\cP(K)$ be fixed and finitely supported.  Then, as $n\to\infty$,
\begin{equation}\label{eq:OP-finite-prior}
|\supp(\widehat G)|=O_{\Prob}(1).
\end{equation}
\end{theorem}

Some further remarks on these results are in order:
\begin{itemize}
    \item The proofs of Theorems 
\ref{thm:continuous-prior} and \ref{thm:main} share the same backbone outlined in \prettyref{sec:proof-outline}, with the latter relying additionally on the adaptive Hellinger rate $H^2(p_{\widehat G}, p_{G_\star}) = O_{\Prob}(1/n)$ due to \cite{ZhangVolgushev2026}.

\item Both theorems concern NPMLE on the hypercube $K=[-M,M]^d$. This choice is primarily for convenience. 
In fact, the bound in \prettyref{thm:continuous-prior} applies to the number of NPMLE atoms in the interior of any compact $K$. The hypercube simplifies the treatment of boundary cases, since each face
is itself a cube of lower dimension.
\end{itemize}

We end this subsection with numerical experiments on NPMLE. We consider three population mixing distributions in dimensions $d=1,2,3$:
(a) a single atom;
(b) two atoms with equal weights;
(c) uniform
distribution. 
\prettyref{fig:support-pmf-three-priors} shows the empirical frequency of the NPMLE support size over 500 independent trials for sample size $n=600$. In these experiments, the fitted model size (in terms of both the mode and the mean) increases with the dimension and with the complexity of the true mixing distribution. For the finite mixture $p_\star$ in (a) and (b), \prettyref{thm:main} shows that $|\supp(\widehat G)|$ is bounded in probability. Determining whether it converges in distribution and how the limiting law  depends on $G_\star$ is an interesting open question.

\begin{figure}[htp]
\centering
\captionsetup[subfloat]{justification=centering}
\subfloat[Single Gaussian: $G_\star=\delta_0$.\label{fig:support-pmf-single-gaussian}]{\includegraphics[width=\textwidth]{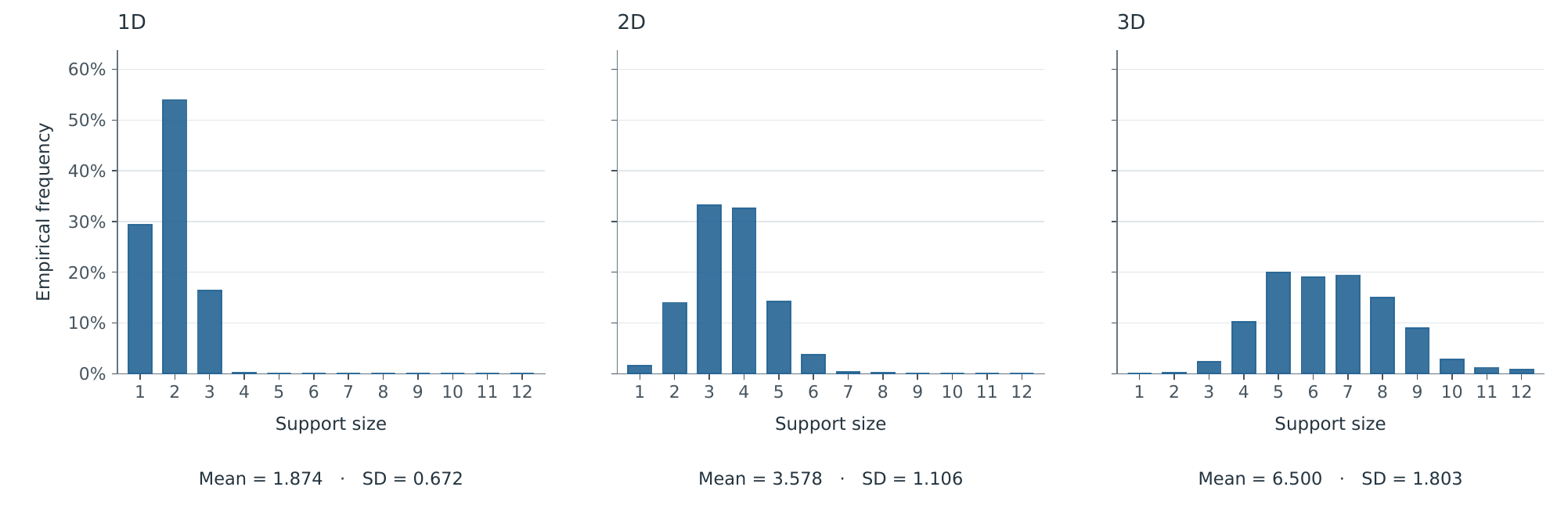}}
\par\medskip
\subfloat[Two-component Gaussian mixture: $G_\star=\frac12\delta_{-\theta_\star}+\frac12\delta_{\theta_\star}$, where $\|\theta_\star\|=1$.\label{fig:support-pmf-two-component}]{\includegraphics[width=\textwidth]{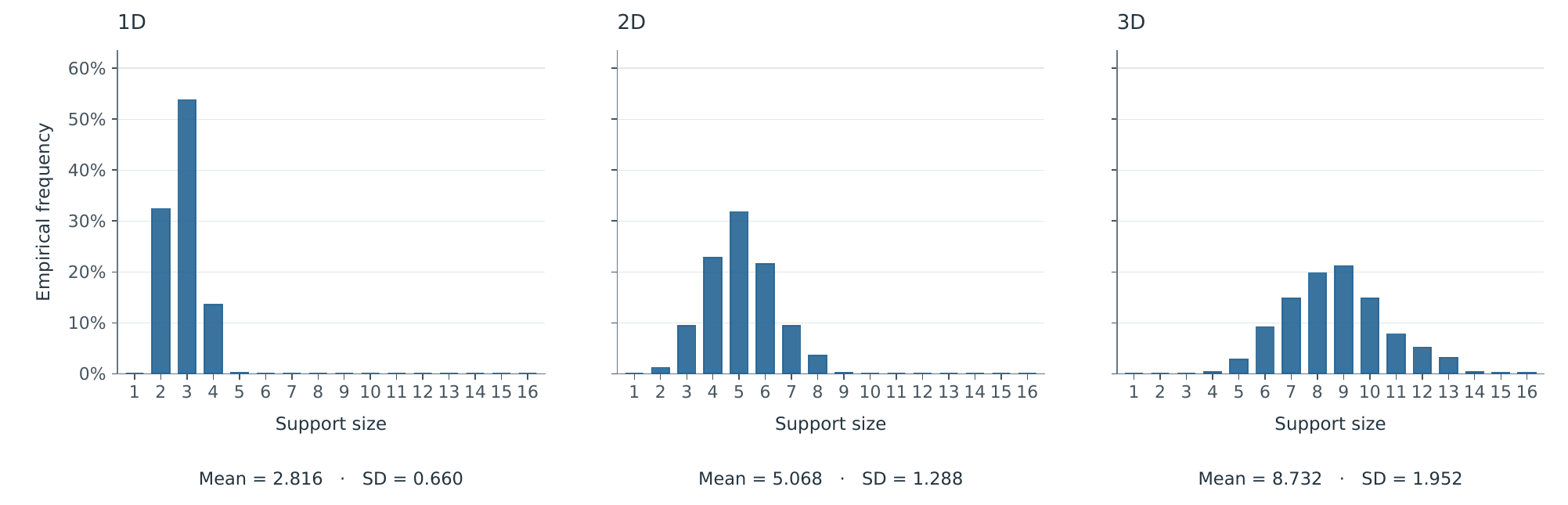}}
\par\medskip
\subfloat[Continuous prior: $G_\star=\operatorname{Unif}((-1,1)^d)$.\label{fig:support-pmf-uniform}]{\includegraphics[width=\textwidth]{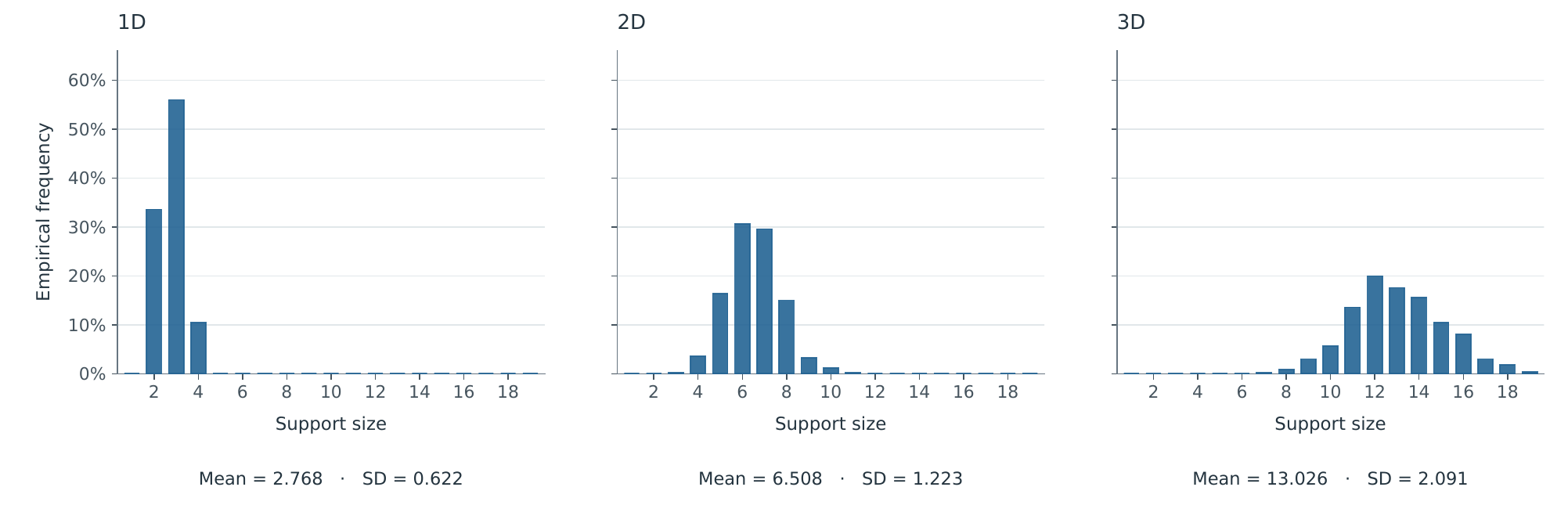}}
\caption{Empirical frequency of the support size of NPMLE computed using the Frank--Wolfe algorithm over the domain $K=[-2,2]^d$ in dimensions $d=1,2,3$. Each panel is based on $T=500$ independent trials with $n=600$ iid observations drawn from $p_\star=G_\star*N(0,I_d)$. The mean and sample standard deviation are shown below each panel. 
We apply a variant of Frank--Wolfe \cite{DenoyelleDuvalPeyreSoubies2020} to speed up convergence.}
\label{fig:support-pmf-three-priors}
\end{figure}

\subsection{Prior and related work}

In terms of proof techniques, both the prior work \cite{PolyanskiyWu2020} in one dimension and the current paper are based on the first-order optimality condition involving the directional derivative of the log likelihood, which we call the \textit{KKT function} (see \prettyref{eq:KKT-function} below), whose global maxima contain the atoms of the NPMLE.
Both works 
use complex-analytic tools to bound the number of critical points of the KKT function in a complex domain containing the support of the NPMLE.
The major difference is that  \cite{PolyanskiyWu2020} applies a one-dimensional Jensen-type inequality (see \prettyref{eq:1DJensen}) to obtain a \textit{deterministic} upper bound of $O((X_{\max}-X_{\min})^2)$, where $X_{\max}$ and $X_{\min}$ are the maximum and minimum of the observations $X_1,\dots,X_n$, respectively. Moreover, as the construction in \cite[Remark 4]{polyanskiy2020note} shows, this deterministic approach provably falls short of showing the stronger $O(\frac{\log n}{\log\log n})$ bound for constrained NPMLE.

Indeed, the proof in the present work is probabilistic, based on decomposing the gradient of the KKT function into two stochastic terms (see  \prettyref{eq:decomposition}). At a technical level, extending the Jensen-type bound on the number of critical points to multiple dimensions (see \prettyref{prop:jensen-critical-count}) requires a high-probability lower bound on the gradient norm \textit{uniformly} on a complex sphere.
In comparison, the one-dimensional Jensen bound in \cite{PolyanskiyWu2020} only needs a lower bound on the KKT derivative at a \emph{single} point, which can be obtained deterministically from monotonicity.

The crux of the proof is bounding the number of critical points of the KKT function, which is a superposition of $n$ Gaussians each centered at a data point (cf.~\prettyref{eq:KKT-function}). This is reminiscent of the results of \cite{mammen1995qualitative,geshkovski2024number} on the number of modes of a kernel density estimator
\begin{equation} \label{eq:KDE}
f_h(\theta) = \frac{1}{nh^d}\sum_{i=1}^n K\pth{\frac{X_i - \theta}{h}}
\end{equation}
where $K$ is a smooth kernel (e.g.~Gaussian) and $h$ is the bandwidth. This problem is amenable to the Kac-Rice formula which expresses the expected number of critical points in terms of the joint distribution of the first and second derivatives of $f_h(\theta)$ (which is approximately Gaussian). This strategy seems difficult to implement here, because unlike \prettyref{eq:KDE} which has the equal weight $\frac{1}{n}$, the KKT function involves the data-driven weights $1/p_{\widehat G}(X_i)$ determined by the solution of the optimization problem \prettyref{eq:likelihood}.

Finally, we mention the very recent work  \cite{jiang2026polylogarithmic} which studies a variant of the NPMLE (\ref{eq:NPMLE-intro}) whose objective function replaces the equal sample weights by random (Gamma) weights, and shows that the associated maximizer has cardinality $O(L_n^d)$ with high probability with respect to both the sample and the random weights. In comparison, the present paper concerns the ordinary NPMLE and 
establishes the $O_{\Prob}(L_n^d)$ support bound in general and $O_{\Prob}(1)$ for finite mixtures. The proof techniques do not appear related.

\paragraph{Notations}

We use $\norm{\cdot}$ without a subscript for Euclidean vector norms and function-space subscripts for function norms, such as $\norm{f}_{L_2(\varphi)}$ and $\norm{f}_{\infty}$. 
Denote the Euclidean ball of radius $R>0$ by 
$B_R=\{z\in\C^d:\norm{z}\leq R\}$ and corresponding sphere by
\begin{equation}
 \mathcal S_R=\{z\in\C^d:\norm{z}=R\}.
 \label{eq:shell}
\end{equation}
Let
$\mathbb D_R^d=\{z\in\C^d:|z_j|\le R,\ 1\le j\le d\}$ be the Cartesian product of $d$ closed disks of radius $R$.

We use the Hellinger convention $H(p,q)^2=\int(\sqrt p-\sqrt q)^2$.
For each integer $s\ge1$, let $\cP_s(K)$ denote the class of mixing distributions supported on at most $s$ points.
Denote the standard normal distribution by $P_0=N(0,I)$ with density $\varphi$, and write $\Expect_0 f=\int f\varphi$.  Define the density ratio
\begin{equation}\label{eq:rG}
r_G(x):=\frac{p_G(x)}{\varphi(x)}
=\int_K\exp\{\theta\cdot x-\norm{\theta}^2/2\} G(\d\theta),
\qquad r_\star=r_{G_\star}.
\end{equation}
Denoting $\gamma_K:=\sup_{\theta\in K}\norm{\theta}$, uniformly over $G\in\cP(K)$, we have
\begin{equation}\label{eq:r-bound}
e^{-\gamma_K\norm{x}-\gamma_K^2/2}\le r_G(x)\le e^{\gamma_K\norm{x}}, \qquad x\in\R^d.
\end{equation}
The Gaussian density \prettyref{eq:kernel} may be extended holomorphically for complex $z \in \complex^d$ as
\begin{equation}\label{eq:pz}
p_z(x)=\varphi(x-z)=(2\pi)^{-d/2}\exp\{-(x-z)\cdot (x-z)/2\},
\end{equation}
where for $z_1,z_2 \in \C^{d}$, $z_1 \cdot z_2$ denotes the complex bilinear product (without conjugate). Define the likelihood ratio of \prettyref{eq:pz} over the true density as:
\begin{equation}\label{eq:Lz}
L_z(x) := \frac{p_z(x)}{p_\star(x)} = \frac{\exp\{z\cdot x-z\cdot z/2\}}{r_\star(x)}.
\end{equation}
Then $\Expect_\star L_z=1$ for every $z\in\C^d$.

\paragraph{AI methodologies statement}
Our proof was developed with the assistance of GPT-5.6 Pro and GPT-6 Pro models. Specifically,  the authors proposed the KKT decomposition \prettyref{eq:decomposition} motivated by local asymptotics and sought a multivariate version of the Jensen-type bound; AI was helpful in identifying a precursor of this result. Building on this program, the current proof strategy emerged after many iterations between the authors and the AI. Notably, the current Hermite approximation argument in \prettyref{sec:shell} is inspired by the first author's ongoing work on adaptive density estimation, and AI suggested applying the same relative error bound (summarized in \prettyref{prop:Hermite-approx}) to sharpen the upper bound in \prettyref{lem:population-shell}. The authors subsequently simplified the proofs substantially, including a clean formulation of \prettyref{prop:jensen-critical-count}, an anti-concentration argument in \prettyref{sec:centered-gradient} relying only on the gradient of the KKT function (without using its value), and a direct small-ball estimate in \prettyref{lmm:esseen} that avoids passing to the Gaussian limit and quantitative CLT. 
AI helped with the simulation for Figs.~\ref{fig:support-pmf-three-priors} and \ref{fig:kkt-flatness}.
The authors wrote the paper and take full responsibility for its contents.

\paragraph{Organization of the paper}
The rest of the paper is organized as follows. Section \ref{sec:KKT} introduces the optimality conditions and the KKT function. Section \ref{sec:proof-outline} presents the key decomposition of the KKT function and outlines the proofs of our main results. 
\prettyref{sec:discussions} provides further discussion of our results and open problems. 
The key ingredients for the proof are established in 
Sections \ref{sec:shell}--\ref{sec:growth} and assembled in \prettyref{sec:proofs} to complete the proofs of \prettyref{thm:continuous-prior} and \prettyref{thm:main}. Atoms on the boundary of the cube are dealt with in \prettyref{app:boundary}. Proofs of the auxiliary results are collected in \prettyref{app:aux}.

\section{KKT function and optimality condition}
\label{sec:KKT}

All structural results on NPMLE, both classical deterministic \cite{simar1976maximum,Jewell1982,Lindsay1983a} and more recent probabilistic results \cite{PolyanskiyWu2020}, build upon  the first-order optimality (KKT) condition, which we first describe. Then, we review a heuristic on local expansion of log likelihood, which motivates our proof technique in \prettyref{sec:proof-outline}.

For $G\in\cP(K)$, define the \emph{KKT function}
\begin{equation}\label{eq:KKT-function}
F_{n,G}(\theta):=\sqrt{n} \Big(\Expect_n \frac{p_\theta}{p_G} - 1\Big)
= \frac{1}{\sqrt{n}} \sum_{i=1}^n \pth{\frac{p_\theta(X_i)}{p_G(X_i)}-1}, \qquad \theta \in K
\end{equation}
which is the normalized version of directional derivative of the log likelihood.
The following lemma relates the atoms of the NPMLE to 
the zeros and critical points of $F_{n,\widehat G}$. See \prettyref{fig:kkt-flatness} for an illustration.

\begin{lemma}\label{lem:KKT}
For the NPMLE $\widehat G$, we have\footnote{Thanks to \prettyref{eq:NPMLE-intro} being a convex optimization problem, the first-order condition \prettyref{eq:KKT} is also sufficient and characterizes the optimality. See, e.g., \cite[Footnote 3]{PolyanskiyWu2020}.}
\begin{equation}\label{eq:KKT}
F_{n,\widehat G} \le 0\quad \text{on }K, \qquad
F_{n,\widehat G} = 0\quad\text{on }\supp(\widehat G).
\end{equation}
Consequently, for every atom $\theta$ of $\widehat G$:
\begin{itemize}
	\item (Interior) If $\theta \in \operatorname{int}K$, then
\begin{equation}\label{eq:interior-contact}
\nabla F_{n,\widehat G}(\theta)=0.
\end{equation}
\item (Boundary) If $\theta$ is in the relative interior of a face of $K$, whose free coordinates are indexed by a nonempty $S\subset [d]$, then $\partial_i F_{n,\widehat G}(\theta)=0$ for all $i \in S$.
\end{itemize} 
\end{lemma}

\begin{proof}
For $\theta\in K$, define
\[
 D_{\widehat G}(\theta) := \Expect_n \frac{p_\theta}{p_{\widehat G}}.
\]
Then $F_{n,\widehat G}=\sqrt n(D_{\widehat G}-1)$. For any $\theta \in K$ and $t\in[0,1]$, consider the perturbation $(1-t)\widehat G + t\delta_\theta\in\cP(K)$, hence optimality of $\widehat G$ implies that
\begin{align*}
\frac{\d}{\d t}\Big|_{t=0^+} \ell_n\big((1-t)\widehat G + t\delta_\theta\big) &= n \big(D_{\widehat G}(\theta) - 1\big) \leq 0.
\end{align*}
Moreover, $\int D_{\widehat G}\d\widehat G=1$. By continuity, we have $D_{\widehat G}=1$ on $\supp(\widehat G)$  and hence \prettyref{eq:KKT} follows.

Finally, an atom $\theta\in\operatorname{int}K$ is a local maximum of $F_{n,\widehat G}$, proving \prettyref{eq:interior-contact}. 
Suppose $\theta$ is in the relative interior of a face, where all coordinates are either $M$ or $-M$ except those in a nonempty $S\subset [d]$. Then the restriction of $F_{n,\widehat G}$ to that face has a local maximum at $\theta$, so $\partial_i F_{n,\widehat G}(\theta)=0$ for all $i \in S$.
\end{proof}

\begin{figure}[ht]
\centering
\captionsetup[subfloat]{justification=centering}
\subfloat[Single Gaussian: $X_i\sim N(0,1)$.\label{fig:kkt-single-gaussian}]{\includegraphics[width=0.48\textwidth]{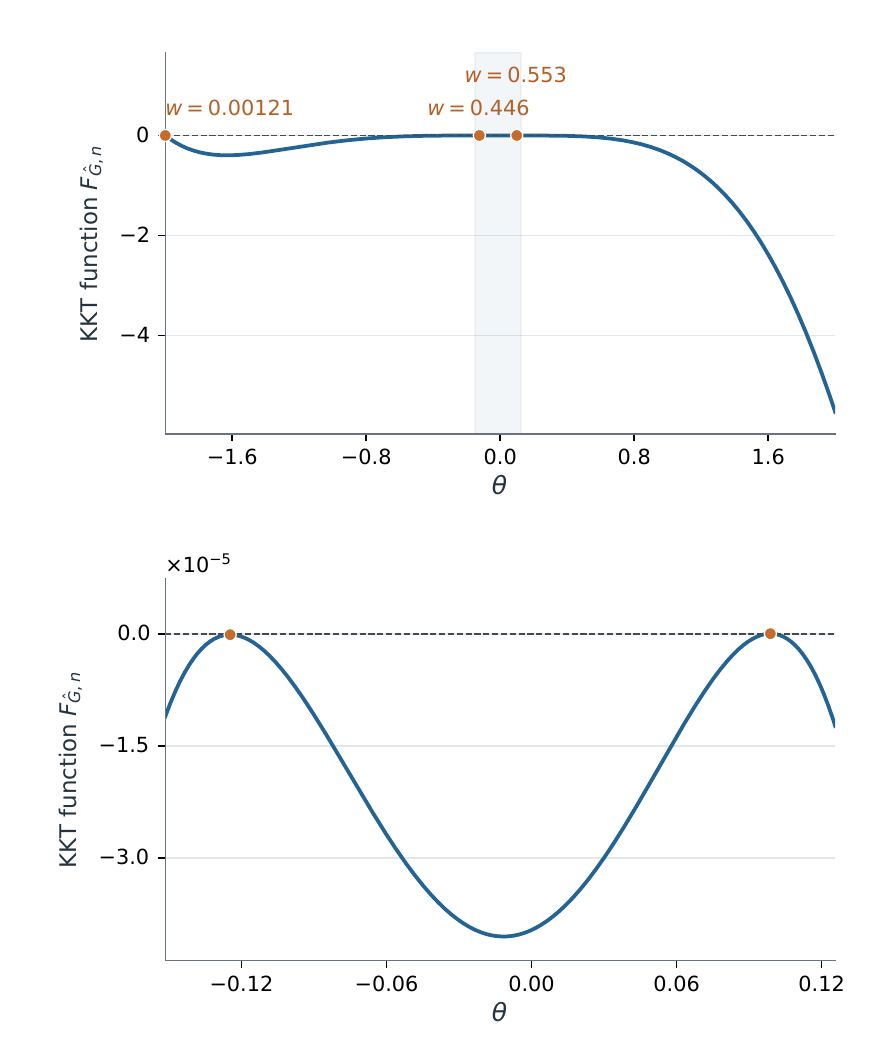}}
\hfill
\subfloat[Two-component Gaussian mixture: \mbox{$X_i\sim\frac12N(-1,1)+\frac12N(1,1)$.}\label{fig:kkt-two-component}]{\includegraphics[width=0.48\textwidth]{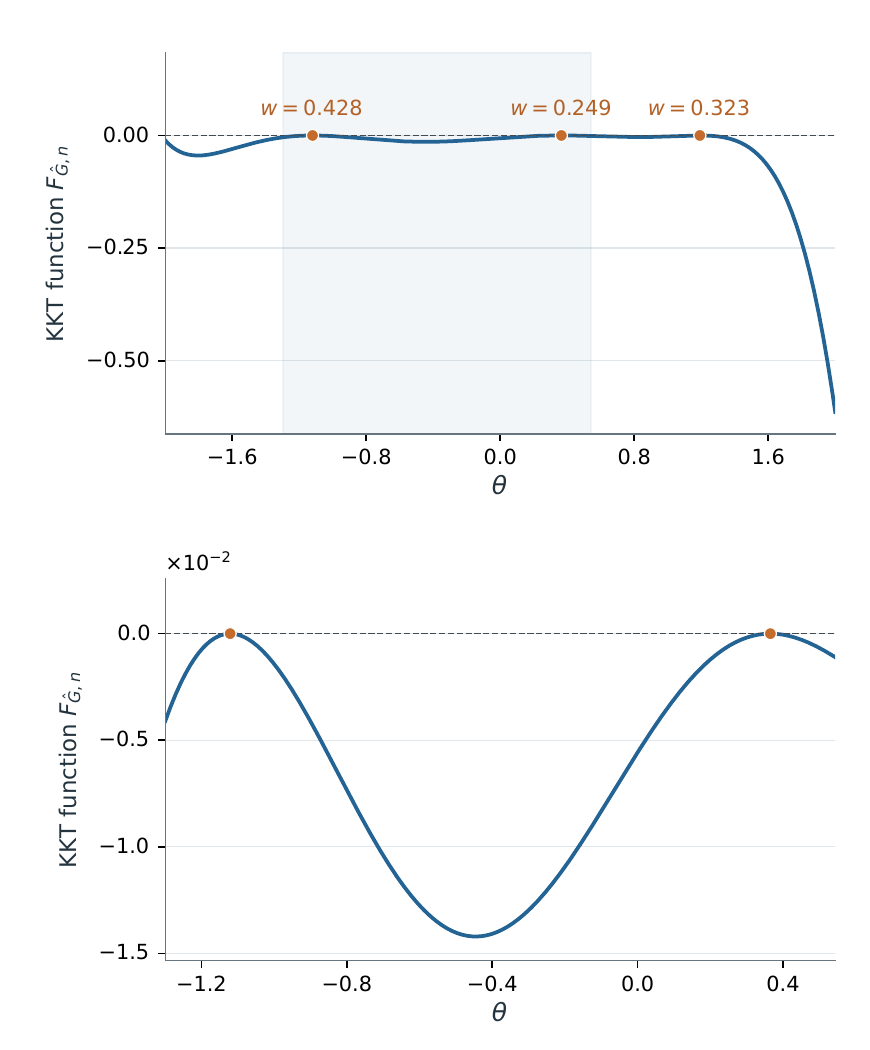}}
\caption{The KKT function $F_{n,\widehat G}$ for NPMLE $\widehat G$ computed using Frank--Wolfe for a sample of $n=10000$ observations in one dimension, together with atoms and weights of the NPMLE $\widehat G$. The upper panel shows the full domain $K=[-2,2]$ and the lower panel zooms into the shaded region, showing that some of the local maxima at the atoms may be exceedingly flat.}
\label{fig:kkt-flatness}
\end{figure}

\subsection{Heuristics from local asymptotics}
\label{sec:lan}

To motivate the proof of our main results, we review the heuristic for the classical
 locally quadratic asymptotics of likelihood (see, e.g., \cite[Lemma 3.1.11]{vdv-wellner}). 
Consider a local perturbation of the true density $p_\star$ of the form
\[
p_{G_{n,h}}=p_\star\left(1+\frac{h}{\sqrt n}\right)
\]
where $h\in \calH\subset L_2(p_\star)$ and $\calH$ consists of all admissible directions: $\Expect_\star h=0$ and $G_{n,h}\in\calP(K)$ exists for all sufficiently large $n$ (concretely, $\calH = \{h=a(p_G/p_\star-1), a\geq 0, G \in \cal P(K)\}$). Under appropriate assumptions, Taylor expansion gives
\begin{align*}
\ell_n(G_{n,h})-\ell_n(G_\star)
&=\sum_{i=1}^n\log\left(1+\frac{h(X_i)}{\sqrt n}\right)\\
&=\sqrt n(\Expect_n-\Expect_\star)h-\frac12\Expect_n h^2+o_{\Prob}(1)\\
&\todistr\ \mathbb{G}(h)-\frac12\norm{h}_{L_2(p_\star)}^2,
\end{align*}
the last step applying CLT and LLN to the first and the second term, respectively. Here $\mathbb{G}(h)$ is a centered Gaussian process indexed by $h\in \calH$, with covariance
$\Cov\big(\mathbb{G}(h_1),\mathbb{G}(h_2)\big)
=\Expect_\star[h_1h_2]$.

The same calculation applies to the KKT function $F_{n,G}$. For fixed $\theta\in \text{int} K$, using $\Expect_\star L_\theta=1$, we obtain
\begin{align*}
F_{n,G_{n,h}}(\theta)
&=\sqrt n\left(\Expect_n\left[\frac{L_\theta}{1+h/\sqrt n}\right]-1\right)\\
&=\sqrt n(\Expect_n-\Expect_\star)L_\theta-\Expect_n[hL_\theta]+o_{\Prob}(1) \numberthis\label{eq:CLT_LLN}\\
&\todistr\ \mathbb{G}(L_\theta-1)-\Expect_\star[hL_\theta].
\end{align*}

As outlined in the next section, our proof applies the same decomposition of the KKT function into a ``CLT term'' and an ``LLN term''. However, 
to count the critical points, we neither seek to establish these asymptotics nor rely on them, as our goal is to obtain results that are applicable to both worst-case and fixed $G_\star$ in \Cref{thm:main,thm:continuous-prior}.

\section{Proof outline}\label{sec:proof-outline}

In this section we outline the proof of our main results. By \prettyref{lem:KKT}, every atom of the NPMLE $\widehat G$ in $\operatorname{int}K$ is a critical point of the corresponding KKT function $F_{n,\widehat G}$, that is, a zero of its gradient. 
To count its critical points, we first need to extend \prettyref{eq:KKT-function} to the complex domain 
\begin{equation*}
F_{n,G}(z):=\sqrt{n} \Big(\Expect_n \frac{p_z}{p_G} - 1\Big), \quad z \in \C^d
\end{equation*}
where $p_z$ is the Gaussian density extended to $\C^d$ given in \prettyref{eq:pz}. The following $d$-dimensional Jensen-type bound, which is essentially a special case of \cite[Theorem 1.7]{BPPSS2024} and proved in \prettyref{app:critical-count} for completeness, controls the number of critical points in terms of the maximum and minimum norm of the gradient on complex spheres.

\begin{proposition}\label{prop:jensen-critical-count}
There exists $C > 0$ only depending on $d$ such that the following holds. Let $R>0$, $\rho\ge2R$, and let $F$ be holomorphic in an open neighborhood of $\mathbb D_\rho^d$.  Define
\begin{equation}\label{eq:jensen-gradient-data}
M_\rho = \sup_{z \in \mathbb D_\rho^d}\norm{\nabla F(z)},
 \qquad m_R = \inf_{z\in\mathcal S_R}\norm{\nabla F(z)}.
\end{equation}
If $m_R>0$, then
\begin{equation}\label{eq:jensen-critical-count}
 \#\{z\in B_R:\nabla F(z)=0\}
 \le C\left[1+
   \frac{\log(M_\rho/m_R)}{\log(\rho/R)}\right]^d,
\end{equation}
where $M_\rho\ge m_R$ because $\mathcal S_R\subset\mathbb D_\rho^d$.
\end{proposition}

We apply \prettyref{prop:jensen-critical-count} with $F=F_{n,\widehat G}$ and $\rho=4R$, choosing $R$ large enough so that $K\subset B_R$. It remains to bound $M_{4R}$ from above and $m_R$ from below. The upper bound for $M_{4R}$ is relatively straightforward: For radius $R$ chosen below, we show in \prettyref{sec:growth} that $
M_{4R}\le e^{CR^2} 
$ with high probability, which follows from the Lipschitz estimate in \prettyref{sec:lip} and the Hellinger density estimation error bounds in \prettyref{sec:shell}. The lower bound for $m_R$ is the main difficulty, and we sketch below the main steps.

The first step is the following decomposition of $F_{n,G}$ that is essentially the same as \prettyref{eq:CLT_LLN}: 
\begin{align}\label{eq:decomposition}
F_{n,G}(z) &= \sqrt{n} \Big(\Expect_n \frac{p_z}{p_G} - 1\Big) \notag \\
&= \sqrt{n} 
\Expect_n\Big(\frac{p_z}{p_\star} - 1 \Big) + \Expect_n \Big(\frac{p_z}{p_\star}\cdot \sqrt{n}\frac{p_\star-p_G}{p_G}  \Big) \notag \\
&= \sqrt{n} 
(\Expect_n-\Expect_\star)\Big(\frac{p_z}{p_\star} \Big) + \Expect_n \Big(\frac{p_z}{p_\star}\cdot \sqrt{n}\frac{p_\star-p_G}{p_G}  \Big) \notag \\
&=: W_n(z) - T_nf_{n,G}(z), 
\end{align}
where we have used that
$L_z=p_z/p_\star$ defined in (\ref{eq:Lz}) satisfies
 $\E_\star[L_z]=\int p_z=1$, and
\begin{align}
     W_n(z) &:=\sqrt n(\Expect_n-\Expect_\star)L_z, \label{eq:Wn} \\
     T_nf(z) &:= \Expect_n [fL_z], \label{eq:Tn} \\
    f_{n,G}(x) &:=\sqrt n\left(1-\frac{p_\star(x)}{p_G(x)}\right)
 =\sqrt n\,\frac{p_G(x)-p_\star(x)}{p_G(x)}.
 \label{eq:fG}
\end{align}
Note that 
$W_n(z) = F_{n,G_\star}(z)$, which can be viewed as the \textit{oracle} version of the finite-sample KKT function $F_{n,\widehat G}(z)$ in \prettyref{eq:KKT-function}, replacing the fitted density $p_{\widehat G}$ in the denominator by the true density $p_\star$.

The key advantage of the decomposition \prettyref{eq:decomposition} is that it allows us to exploit the stochastic nature of $F_{n,\widehat G}$, which we describe in two parts:
\begin{itemize}
\item Using the heuristic (\ref{eq:CLT_LLN}), the first term $W_n$, representing the empirical fluctuation of $L_z$, converges in distribution to a Gaussian process as $n\to\infty$.
Consequently, after normalization by the energy of $L_z$ (see details below), we expect that, with high probability, the gradient $\nabla W_n(z)$ exhibits no substantial decay as $\|z\|$ increases.
\item In contrast, the second term $T_n f_{n,\widehat G}$ can be controlled using the statistical properties of the NPMLE $\widehat G$.
Specifically, the Gaussian mixture structure, together with the Hellinger closeness
$p_{\widehat G}\approx p_\star$, yields a
low-degree Hermite approximation of the relative
density estimation error, with a rapidly decaying
remainder. This allows us to show that the normalized gradient of $T_nf_{n,\widehat G}$ decays rapidly as $\|z\|$ increases.
\end{itemize}
Together, these observations yield a lower bound on $m_R$ for sufficiently large $R$. 
This decomposition \prettyref{eq:decomposition} marks a sharp departure from the approach of \cite{PolyanskiyWu2020}, which studies the same KKT function in one dimension, but mostly from a deterministic perspective.

We provide more details following the decomposition \prettyref{eq:decomposition}. It will be more convenient to work with the following normalized version of the gradient:
\begin{equation}\label{eq:gradient-field}
D_zF := w_z\nabla F(z),
\end{equation}
where, with $r_\star$ given by (\ref{eq:rG}),
\begin{equation}
w_z :=\sqrt{r_\star(2\Rea z)}e^{-\norm{z}^2/2}.
\label{eq:wz}
\end{equation}
Here the normalization $w_z$ is chosen to account for the energy of $L_z$ (i.e., the exponential growth of the second moment of $L_z$.) Indeed, writing $z=a+ib$, we have
\[
\E_\star|L_z|^2
=e^{\norm{z}^2}\int_{\R^d}
\frac{\varphi(t)}{r_\star(2a+t)}\d t
\asymp
\frac{e^{\norm{z}^2}}{r_\star(2a)} = \frac{1}{w_z^2},
\]
where the middle step is justified in Lemma \ref{lem:weight-bounds} below. 
Using \prettyref{eq:r-bound}, the sought-after infimum $m_R$ of $\|\nabla F\|$ is related to that of the normalized $\|D_zF\|$ as follows:
\begin{align}\label{eq:normalization_connection}
m_R \ge e^{R^2/2-\gamma_KR}\inf_{z\in \calS_R} \norm{D_zF}. \end{align}
It remains to lower bound $\inf_{z\in \mathcal S_R} \norm{D_zF}$. Using the decomposition \prettyref{eq:decomposition} and the triangle inequality,
\begin{equation}
\inf_{z\in\mathcal S_R}\norm{D_zF_{n,G}} \ge\inf_{z\in\mathcal S_R}\norm{D_zW_n}
   -\sup_{z\in\mathcal S_R}\norm{D_zT_nf_{n,G}}.
 \label{eq:gradient-triangle}
\end{equation}
We then bound the two terms separately:
\begin{enumerate}[label=(\roman*),leftmargin=*]
\item \emph{Upper bound for the density estimation error term.}  In \prettyref{sec:shell}, we first prove a \emph{deterministic} bound for the normalized gradient of its population counterpart $Tf_{n,G}$, defined as
\begin{equation} \label{eq:Tn-T}
Tf(z) := \Expect_\star[fL_z] = \int_{\R^d}f(x)p_z(x)\d x.
\end{equation}
Our bound states the following: for every $G_\star\in\cP(K)$, $h>1$, and large $R$:
\begin{align}\label{eq:population_correction_general}
\sup_{\substack{G\in\mathcal P(K)\\ \sqrt n\,H(p_G,p_\star)\le h}} \sup_{z\in \mathcal S_R}\norm{D_zTf_{n,G}}
&\leq Ch\,e^{-cR} + C\sqrt n\,e^{-cR^2\log R}.
\end{align}
Furthermore, if $G_\star\in\cP_s(K)$ and $R\ge C_{K,s}$, for a constant $C_{K,s}$ depending only on $K,s$, then
\begin{align}\label{eq:population_correction_finite}
\sup_{\substack{G\in\mathcal P(K)\\
                 \sqrt n\,H(p_G,p_\star)\le h}}
\sup_{z\in \mathcal S_R}\norm{D_zTf_{n,G}} &\leq Ch\,e^{-cR}.
\end{align}
For arbitrary $G_\star\in\cP(K)$, \prettyref{lem:universal-H} gives $\sqrt nH(p_{\widehat G},p_\star)\le h$ with $h=C_K(\log n)^{(d+1)/2}$ and failure probability at most $n^{-c}$, uniformly over $G_\star$. For a fixed finitely supported $G_\star$, the parametric Hellinger bound in \cite[Theorem~2.1(1)]{ZhangVolgushev2026} gives $\sqrt nH(p_{\widehat G},p_\star)=O_{\Prob}(1)$. 
Thus a sufficiently large deterministic $h$, independent of $n$, makes the probability of $\sqrt nH(p_{\widehat G},p_\star)> h$ arbitrarily small for all sufficiently large $n$.

We then bound the normalized gradient of the empirical approximation error $(T_n-T)f_{n,G}$:  If $\log n\ge C(R^2+\log h)$, then, with probability at least $1-n^{-c}$,
\[
 \sup_{\substack{G\in\cP(K)\\\sqrt n\,H(p_G,p_\star)\le h}}
 \sup_{z\in \mathcal S_R}\norm{D_z(T_n-T)f_{n,G}}\le n^{-c}.
\]
Since $T_n=T+(T_n-T)$, the corresponding bound for $\norm{D_zT_nf_{n,G}}$ is the population bounds (\ref{eq:population_correction_general}) and (\ref{eq:population_correction_finite}) plus $n^{-c}$ on this event.  For the NPMLE, the additional failure probability is the probability that $\sqrt nH(p_{\widehat G},p_\star)>h$; \prettyref{prop:fitted-shell} states the resulting estimate.

\item \emph{Lower bound for the fluctuation.}  In \prettyref{sec:centered-gradient}, we prove the \emph{probabilistic} estimate that
\begin{align}\label{eq:empirical-field}
 \inf_{z\in \mathcal S_R} \norm{D_zW_n} \ge e^{-cR}
\end{align}
with probability at least $1-e^{-c'R}$, provided that $n\geq e^{CR^2}$. The proof combines a pointwise anti-concentration argument with a union bound. The high-level idea is as follows:
For a typical $z\in \mathcal S_R$, \prettyref{lem:point-small-ball} gives a small-ball probability bound $\Prob(\norm{D_zW_n}\le \eps)=O(\eps^{2d})$, consistent with the CLT intuition that $D_zW_n$ behaves as a $d$-dimensional complex Gaussian vector.
Using the typical Lipschitz continuity of $z\mapsto D_zW_n$ established in \prettyref{lem:sphere-variation}, we extend this pointwise bound uniformly over $\mathcal S_R$ via an $\eps$-net of cardinality $(R/\eps)^{2d-1}$.
\textit{Crucially}, the small-ball exponent $2d$ exceeds the dimension $2d-1$ of the sphere $\mathcal S_R$, leaving a factor of $\eps$ after the union bound.
This argument is formalized in \prettyref{prop:centered-gradient-lower}. 
\end{enumerate}

Combining the gradient lower bound \prettyref{eq:empirical-field} and the gradient upper bound \prettyref{eq:population_correction_general} or \prettyref{eq:population_correction_finite} with a suitable Hellinger threshold $h$, we conclude from \prettyref{eq:gradient-triangle} that 
\[
\inf_{z\in\mathcal S_R}\norm{D_zF_{n,\widehat G}} \ge \frac{1}{2}e^{-cR}
\]
with high probability. 
This holds for $R=O(\sqrt{\frac{\log n}{\log\log n}})$ when $G_\star\in \calP(K)$ is arbitrary, and for $R=O(1)$ when $G_\star$ is $s$-atomic.
By \prettyref{eq:normalization_connection}, this implies $m_R\ge e^{cR^2}$ for sufficiently large $R>0$.
Applying the Jensen estimate in \prettyref{prop:jensen-critical-count} together with the high-probability bound $M_{4R}\le e^{CR^2}$, we obtain the desired upper bound of $O(R^{2d})$ on the number of critical points in $\text{int} K$.
The number of critical points on the boundary $\partial K$ is controlled in \prettyref{app:boundary} using similar arguments and bounded by $O(R^{2(d-1)})$. 
Finally, the choice $R=O(\sqrt{\frac{\log n}{\log\log n}})$ yields \prettyref{thm:continuous-prior}, while the choice $R=O(1)$ yields \prettyref{thm:main}.

\section{Discussion}\label{sec:discussions}

In this section we compare the adaptation property of NPMLE for Gaussian mixtures with  
other nonparametric models and conclude with open questions.

\paragraph{Structural adaptation in other nonparametric models.} 

At a high level, adaptation concerns whether an estimator can exploit
simplicity in the true model without knowing its complexity in advance.
Statistical adaptation means achieving the estimation rate available
in that simpler model, while structural adaptation means that the fitted
model itself has low complexity.
For Gaussian mixtures, these two forms of adaptation go hand in hand:
when $G_\star$ is fixed and finitely supported, the NPMLE achieves
the parametric rate
$H^2(p_{\widehat G},p_{G_\star})=O_{\Prob}(1/n)$
by \cite[Theorem~2.1(1)]{ZhangVolgushev2026}, and fits only $O_{\Prob}(1)$ components
by \prettyref{thm:main}.
The connection between these two forms of adaptation can be seen
through the log-likelihood: Statistical bounds control its fluctuations
near the population, whereas structural bounds require more nuanced understanding 
of its directional derivative (KKT function) at the maximizer.
Such strong simultaneous adaptation is not universal among nonparametric
models, as the following example illustrates.

Consider the class of monotonically decreasing densities on $[0,1]$, or  equivalently, the family of scale mixtures of uniform distributions: $p_\theta = \rm Unif[0,\theta]$ for $\theta > 0$ and the mixing distribution is supported on $K =[0,1]$. In this case, the NPMLE is the celebrated Grenander estimator \cite{grenander1956theory}, which is known to be piecewise-constant, and the analogous model complexity parameter is its number of pieces. Even when $G_\star = \delta_1$ is a point mass (so $p_\star = \rm Unif[0,1]$), \cite{andersen1954fluctuations} (see also \cite[Theorem 4]{groeneboom2021grenander}) shows that the number of fitted pieces is $\Theta_{\Prob}(\log n)$, with a limiting normal distribution after centering and scaling, rather than $O_{\Prob}(1)$. This difference can be seen from two perspectives. Statistically, \cite[Theorem 2]{groeneboom2021grenander} 
shows that the Grenander estimator achieves a squared Hellinger density estimation rate $O_{\Prob}(\frac{\log n}{n})$, slower than the parametric rate.
Structurally, the oracle version of the KKT function in \prettyref{eq:Wn}, namely 
$W_n(\theta) = \sqrt{n}(\E_n \frac{p_\theta}{p_\star}-1)$, converges weakly to the random function $\frac{\mathbb B(\theta)}{\theta}$ on $\theta\in (0,1)$, where $\mathbb B$ denotes the Brownian bridge. 
This limiting function is almost surely nowhere differentiable with infinitely many local maxima in $(0,1)$, reflecting the roughness of the finite-sample KKT function $F_{n,\widehat G}$. In contrast, the corresponding limiting function for Gaussian mixtures is almost surely analytic, consistent with the regularity
of $F_{n,\widehat G}$ that plays a crucial role in the proof of our main results.

\paragraph{Further questions.} 
\begin{itemize}
\item 
Can analogous self-regularization and structural adaptation results
be established for multivariate mixtures in other exponential families,
extending the one-dimensional results of \cite{PolyanskiyWu2020}? 
A notable example is the class of Poisson mixtures, where the critical points of the KKT function are determined by a polynomial system. In this case \cite{PolyanskiyWu2020} proves an  $O_\Prob(\log n)$ cardinality bound for the NPMLE when the true prior is sub-exponential. We expect the proof technique in the present paper can be adapted to the Poisson setting to prove the same $O_{\Prob}(L_n^d)$ and $O_{\Prob}(1)$ bound in the respective settings of Theorems \ref{thm:continuous-prior} and \ref{thm:main}, but this will be pursued in a separate work.
\item A more challenging problem concerns structural properties of the \emph{unconstrained} NPMLE. While the multivariate Jensen-type bound in Proposition \ref{prop:jensen-critical-count} is still applicable, a priori we can only localize the support of the NPMLE to the convex hull of the sample, contained in a ball of radius $\max_{i\in[n]}\pnorm{X_i}{}$. This requires us to lower bound the KKT gradient on a much larger radius of $O(\sqrt{\log n})$ (as opposed to the  $O(\sqrt{\log n/\log\log n})$ for \prettyref{thm:continuous-prior}). Under the current decomposition \prettyref{eq:decomposition} of the KKT function into an empirical fluctuation term and a density estimation error term, many of the analytic estimates for these terms are difficult to generalize directly. Partly due to this technical obstacle, we expect the self-regularization and structural adaptation of unconstrained NPMLE to be weaker.\footnote{\label{fn:support-divergence}In fact, the AI-generated manuscript available at \url{https://yangpliu.github.io/repo/support-divergence/paper.pdf} claims that, even when $G_\star$ is a point mass, the support size of the unconstrained NPMLE diverges logarithmically with high probability.}
\end{itemize}

\section{Gradient upper bound for the density estimation error term}\label{sec:shell}

In this section, we upper bound the supremum in \prettyref{eq:gradient-triangle} as a function of the density estimation error in Hellinger distance. For a given true prior $G_\star$, let
\begin{equation}\label{eq:fitted-hellinger-failure}
 a_n(h) := \Prob_{G_\star}\Big(\sqrt n\,H(p_{\widehat G},p_\star)>h
 \Big). 
\end{equation}

\begin{proposition}\label{prop:fitted-shell}
There are constants $C,c,R_0,n_0>0$, depending only on $K$, such that the following holds.  Let $G_\star\in\cP(K)$, $h>1$, $R\ge R_0$, and $n\ge n_0$ satisfy
\begin{equation}
 \log n\ge C(R^2+\log h).
 \label{eq:fitted-sample-condition}
\end{equation}
With probability at least $1-a_n(h)-n^{-c}$,
\begin{equation}\label{eq:fitted-shell}
 \sup_{z\in\mathcal S_R}
 \norm{D_zT_nf_{n,\widehat G}}
 \le Ch e^{-cR}+C\sqrt n\,e^{-cR^2\log R}+n^{-c}.
\end{equation}
Furthermore, for each integer $s\ge1$ there is a constant $C_{K,s}>0$, depending only on $K,s$, such that if $G_\star\in\cP_s(K)$ and $R\ge C_{K,s}$, then, with the same probability guarantee,
\begin{equation}
 \sup_{z\in\mathcal S_R}
 \norm{D_zT_nf_{n,\widehat G}}
 \le Ch e^{-cR}+n^{-c}.
 \label{eq:fitted-shell-atomic}
\end{equation}
\end{proposition}

\begin{remark}[Role of Hellinger localization]
\label{rmk:hellinger-localization}
We comment on the role of the Hellinger upper bound $\sqrt{n}H(p_{\widehat G},p_\star)\le h$ in the proof of \prettyref{prop:fitted-shell}. Roughly speaking, $\sqrt{n}H(p_{\widehat G},p_\star)$ corresponds to the $L_2(p_\star)$-norm of the relative error $f_{n,\widehat G}$ in \prettyref{eq:fG}. One might therefore expect that an upper bound $\E_\star u^2 \le h^2$ would, in general, imply a uniform decay of $D_zT_nu$ on the sphere $\calS_R$ as $R$ increases. However, this implication fails even for $d=1$ and the population case $n=\infty$.
For example, when $p_\star=\varphi$, take an integer $r=R^2\ge1$, and
$u(x)=h\,H_r(x)/\sqrt{r!}$, where $H_r$ is the Hermite polynomial defined below.
Then $\Expect_\star u^2=h^2$ and $Tu(z)=h\,z^r/\sqrt{r!}$.
For every $z\in \C$ with $|z|=R$, Stirling's approximation gives
\[
 \norm{D_zTu}
 =h\,e^{-R^2/2}\frac{rR^{r-1}}{\sqrt{r!}}
 \asymp h\sqrt{R},
\]
which is growing in $R$. In other words, the Hellinger localization by itself is not enough to ensure the exponential decay $h e^{-cR}$ in \prettyref{eq:fitted-shell-atomic}. 

By contrast, our proof heavily relies on the additional structure of the function $f_{n,G}$ in (\ref{eq:fG}). For example, when $G_\star$ is $s$-atomic, \prettyref{prop:Hermite-approx} gives, for $m\ge C_{K,s}$, 
\[
 r_G-r_\star=P_m+R_m,\qquad \deg P_m\le m,\qquad
 \norm{R_m}_{L_2(\varphi)}\le C_KH(p_G,p_\star) e^{-c_Km\log m}.
\]
Consequently, we can write
$f_{n,G}=\sqrt n(P_m+R_m)/r_G$.
In the proof of \prettyref{lem:population-shell}, we will choose $m=\kappa R^2$ for a sufficiently small fixed constant $\kappa>0$.
Since $\kappa<1$, the low-degree component $P_m$ avoids the obstruction illustrated by the preceding counterexample, which involves a polynomial of degree $R^2$.
Meanwhile, the high-degree remainder $R_m$ will be controlled using its $L_2(\varphi)$ norm.
\end{remark}

We first establish the needed Hermite approximation in Section \ref{sec:atomic-approx} to bound the normalized gradient of the population counterpart $Tf_{n,G}$ in Section \ref{subsec:fitted_correction_pop}.  We then control in Section \ref{subsec:fitted_correction_err} the normalized gradient of the empirical approximation error $(T_n-T)f_{n,G}$ uniformly over the Hellinger neighborhood.  Combining these two estimates proves the proposition.

\subsection{Hermite approximation}
\label{sec:atomic-approx}

Let $H_j$ be the degree-$j$ Hermite polynomial orthogonal under the scalar standard normal density with $\int_\reals H_iH_j \varphi = i!\Indc\{i=j\}$. In  multiple dimensions, for a multi-index $\alpha=(\alpha_1,\ldots,\alpha_d)\in\integers_+^d$, let
\[
 H_\alpha(x):=\prod_{j=1}^dH_{\alpha_j}(x_j),
 \qquad \psi_\alpha:=\frac{H_\alpha}{\sqrt{\alpha!}}
\]
where $\alpha!=\prod_{j=1}^d\alpha_j!$. Then $(\psi_\alpha)$ is an orthonormal polynomial basis for $L_2(\varphi)$.  For $h\in L_2(\varphi)$, the truncated Hermite expansion
\[
 \Pi_m h:=\sum_{|\alpha|\le m}
       \langle h,\psi_\alpha\rangle_{L_2(\varphi)}\psi_\alpha
\]
is the best degree-$m$ polynomial approximation of $h$ in $L_2$.

Define the error of likelihood ratios by $\Delta r_G:=r_G-r_\star$, with $r_G$ and $r_\star$ given by (\ref{eq:rG}). The main result of this subsection bounds its polynomial approximation and the remainder of $\Delta r_G$ in both absolute and relative sense. While the absolute error bound is well-known as the Hermite coefficients are given by moment differences, the relative bound proportional to 
$H(p_G,p_\star)$ is the key which relates to the Hellinger distance between $p_G$ and $p_\star$.

\begin{proposition}\label{prop:Hermite-approx}
There are constants $C_K,c_K>0$, depending only on $K$, such that for every $G,G_\star\in\cP(K)$ and integer $m\ge2$,
\begin{equation}
 \norm{\Pi_m(\Delta r_G)}_{L_2(\varphi)}
 \le C_K H(p_G,p_\star)e^{C_K\sqrt m},
 \label{eq:low-poly}
\end{equation}
and the remainder satisfies the following absolute bound:
\begin{equation}\label{eq:absolute-tail}
 \norm{\Delta r_G-\Pi_m(\Delta r_G)}_{L_2(\varphi)}
 \le C_Ke^{-c_Km\log m}.
\end{equation}

Furthermore, for each integer $s\ge1$ there is a constant $C_{K,s}>0$, depending only on $K,s$, such that if $G_\star\in\cP_s(K)$ and $m\ge C_{K,s}$, then
the remainder satisfies the following relative bound:
\begin{equation}
 \norm{\Delta r_G-\Pi_m(\Delta r_G)}_{L_2(\varphi)}
 \le C_K H(p_G,p_\star)e^{-c_Km\log m}.
 \label{eq:relative-tail}
\end{equation}
\end{proposition}

The proof of \prettyref{prop:Hermite-approx} uses the Hermite coefficients of $\Delta r_G$. Recall the Hermite generating identity 
\begin{equation}
 e^{z\cdot x-z\cdot z/2}
 =\sum_{\alpha\in\Nzero^d}\psi_\alpha(x)\frac{z^\alpha}{\sqrt{\alpha!}}.
 \label{eq:Hermite-generating}
\end{equation}
Consequently, the density-ratio error has the expansion
\begin{equation}
 \Delta r_G(x)=\sum_\alpha \Delta_\alpha(G)\frac{H_\alpha(x)}{\alpha!}
 =\sum_\alpha\frac{\Delta_\alpha(G)}{\sqrt{\alpha!}}\psi_\alpha(x), \qquad \Delta_\alpha(G):=\int_K\theta^\alpha \d(G-G_\star)(\theta).
 \label{eq:ratio-error-Hermite}
\end{equation}
The following lemma provides two estimates of the moment difference $\Delta_\alpha(G)$ in terms of $H(p_G,p_\star)$, for small and large $|\alpha|:=\sum_{j=1}^d \alpha_j$, respectively. 

\begin{lemma}\label{lem:polarization}
There is a constant $C_K>0$, depending only on $K$, such that for every $G,G_\star\in\cP(K)$ and every multi-index $\alpha\in\Nzero^d$,
\begin{equation}
 \frac{|\Delta_\alpha(G)|}{\sqrt{\alpha!}}
 \le C_K H(p_G,p_\star)e^{C_K\sqrt{|\alpha|}}.
 \label{eq:low-coeff}
\end{equation}
For each integer $s\ge1$, there is a constant $C_{K,s}>0$, such that if $G_\star\in\cP_s(K)$, then for $|\alpha|=\ell>2s$,
\begin{equation}
 |\Delta_\alpha(G)|
 \le C_{K,s}^{\ell+1}H(p_G,p_\star).
 \label{eq:mixed-high}
\end{equation}
\end{lemma}
The proof of \prettyref{lem:polarization} is deferred to Appendix~\ref{app:mixed-moments}. In fact, the second bound follows from the first by applying the moment-comparison result in \cite[Lemma~6.3]{ZhangVolgushev2026}, which states that higher order moment differences can be controlled by the first $2s$ ones. We will use the two estimates in \prettyref{lem:polarization} to control the polynomial part and, when $G_\star$ is atomic, the relative remainder.

\begin{proof}[Proof of \prettyref{prop:Hermite-approx}]
Write $H=H(p_G,p_\star)$.  By \prettyref{lem:polarization}, the coefficient bound \prettyref{eq:low-coeff} holds for arbitrary $G,G_\star\in\cP(K)$. 
By orthogonality of the Hermite polynomials, we get
\[
 \norm{\Pi_m(\Delta r_G)}_{L_2(\varphi)}^2
 =\sum_{|\alpha|\le m}\frac{|\Delta_\alpha(G)|^2}{\alpha!}
 \le C_KH^2(m+1)^de^{C_K\sqrt m}.
\]
Here we used $\#\{\alpha\in\Nzero^d:|\alpha|\le m\}=\binom{m+d}{d}\le(m+1)^d$.  Taking square roots and absorbing $(m+1)^{d/2}$ into $e^{C_K\sqrt m}$ proves \prettyref{eq:low-poly}.

For the absolute remainder bound, since \(G\) and \(G_\star\) are supported on \(K\), we have
$|\Delta_\alpha(G)|\le2\gamma_K^{|\alpha|}$.  We then get from orthogonality and $\sum_{|\alpha|=\ell}1/\alpha!=d^\ell/\ell!$ that
\[
 \norm{\Delta r_G-\Pi_m(\Delta r_G)}_{L_2(\varphi)}^2=\sum_{\ell>m}\sum_{|\alpha|=\ell}
   \frac{|\Delta_\alpha(G)|^2}{\alpha!}\le4\sum_{\ell>m}\frac{(d\gamma_K^2)^\ell}{\ell!},
\]
which, by Stirling's approximation, yields \prettyref{eq:absolute-tail} for all $m\ge C_{K}$, provided that $C_{K}$ is sufficiently large. The relative bound \prettyref{eq:relative-tail} follows from the same calculation, where we use \prettyref{eq:mixed-high} in \prettyref{lem:polarization} instead. 
\end{proof}

\subsection{Population version}\label{subsec:fitted_correction_pop}

With the convention \prettyref{eq:Tn-T}, the population version of $T_nf_{n,G}$ is
\begin{equation}
 Tf_{n,G}(z)=\sqrt n\,\Expect_0\left\{
 \frac{\Delta r_G(X)}{r_G(X)}e^{z\cdot X-z\cdot z/2}\right\}.
 \label{eq:Tf-formula}
\end{equation}
The following lemma establishes the population version of \prettyref{prop:fitted-shell}.
\begin{lemma}\label{lem:population-shell}
There are constants $R_0>1$ and $c,C>0$, depending only on $K$, such that the following holds.  Let $G,G_\star\in\cP(K)$, $n\ge1$, $R\ge R_0$, and $h>1$ satisfy
\begin{equation}
 \sqrt n\,H(p_G,p_\star)\le h.
 \label{eq:H-ball}
\end{equation}
Then
\begin{equation}
 \sup_{z\in\mathcal S_R}\norm{D_zTf_{n,G}}
 \le Ch e^{-cR}+C\sqrt n\,e^{-cR^2\log R}.
 \label{eq:population-shell}
\end{equation}

In addition, for each integer $s\ge1$, there is a constant $C_{K,s}>0$, depending only on $K,s$, such that if $G_\star\in\cP_s(K)$ and $R\ge C_{K,s}$, then
\begin{equation}\label{eq:population-shell-atomic}
\sup_{z\in\mathcal S_R}\norm{D_zTf_{n,G}}
 \le Ch e^{-cR}.
\end{equation}
\end{lemma}

We prove \prettyref{lem:population-shell} in the rest of this section. We start with a simple lemma for bounding the denominator.
\begin{lemma}
\label{lem:zero-free}
There is $\tau_K>0$ such that for every $G\in\cP(K)$,
\begin{equation}
 |r_G(x+iy)|\ge\frac12r_G(x),
 \qquad x,y\in\R^d,\quad \norm{y}\le\tau_K.
 \label{eq:zero-free}
\end{equation}
\end{lemma}

\begin{proof}
Choose $\tau_K$ so that $|\theta\cdot y|\le\pi/3$ for $\theta\in K$ and $\norm{y}\le\tau_K$. Then
\begin{align*}
\frac{r_G(x+iy)}{r_G(x)} &= \frac{\varphi(x)}{\varphi(x+iy)}\cdot \frac{\int_K e^{-(x-\theta+iy)\cdot (x-\theta+iy)/2}\d G(\theta)}{\int_K e^{-\norm{x-\theta}^2/2} \d G(\theta)} \\
&= \frac{\int_K e^{-\norm{x-\theta}^2/2 + i\theta\cdot y}\d G(\theta)}{\int_K e^{-\norm{x-\theta}^2/2} \d G(\theta)}
\end{align*}
has real part at least $1/2$, implying \prettyref{eq:zero-free}.  
\end{proof}

We also need a bound on complex translations of the polynomial part.  The following lemma is proved in Appendix~\ref{app:polynomial-translations}.

\begin{lemma}\label{lem:Gaussian-translation}
Let $m\ge 0$ be an integer, and $P$ be a polynomial in $d$ variables with complex coefficients and total degree at most $m$.  For every $z\in\C^d$,
\begin{equation}
 \norm{P(\cdot+z)}_{L_2(\varphi)}
 \le e^{\norm{z}\sqrt m}\norm P_{L_2(\varphi)}.
 \label{eq:Gaussian-translation}
\end{equation}
\end{lemma}

\begin{proof}[Proof of \prettyref{lem:population-shell}]
Choose a sufficiently small constant $\kappa>0$, depending only on $K$, as specified below, and set $m=\lfloor\kappa R^2\rfloor$.  Increasing $R_0$ ensures that for $R \geq R_0$,
\[
 2\le\frac\kappa2R^2\le m\le\kappa R^2.
\]
Write $\Delta r_G=P_m+R_m$ with $P_m=\Pi_m(\Delta r_G)$. In the atomic case, assume in addition that $R\ge C_{K,s}$, with this radius threshold chosen large enough that $m=\lfloor\kappa R^2\rfloor$ satisfies the degree threshold in \prettyref{prop:Hermite-approx}. With $H=H(p_G,p_\star)$, \prettyref{prop:Hermite-approx} gives
\begin{align}
 \norm{P_m}_{L_2(\varphi)}
 &\le C_KH e^{C_K\sqrt m}, \label{eq:polynomial-bound} \\
  \norm{R_m}_{L_2(\varphi)}
 &\le \begin{cases}C_Ke^{-c_Km\log m} &\text{if } G_\star \in \calP(K), \\
 C_K He^{-c_Km\log m} &\text{if } G_\star\in \calP_s(K). 
 \end{cases} \label{eq:remainder-bound}
\end{align}
Write $z=a+ib$, with $a=\Rea z$ and $b=\Ima z$. We bound the remainder and polynomial parts separately. \\

\noindent \emph{Contribution of the remainder.} For the remainder, we prove that
\begin{equation}
\sqrt n\sup_{z\in \mathcal S_R}\norm{D_zT(R_m/r_G)} \le 
 \begin{cases}
     C\sqrt{n}e^{-cR^2\log R} &\text{if } G_\star\in \calP(K), \\
     Che^{-cR^2\log R} &\text{if } G_\star\in \calP_s(K)\text{ and } \sqrt{n}H(p_G,p_\star)\le h.
 \end{cases}
 \label{eq:scaled-tail-shell-bound}
\end{equation}
To prove \prettyref{eq:scaled-tail-shell-bound}, recall the definition of $T$ in \prettyref{eq:Tf-formula} and 
the normalization in \prettyref{eq:gradient-field}.
Applying the Gaussian derivative identity 
$\nabla_zp_z(x)=(x-z)p_z(x)$, we have
\begin{align}\label{eq:DzT-formula}
D_z T(R_m/r_G) = \sqrt{r_\star(2a)} \Expect_0\left\{
   \frac{R_m(X)}{r_G(X)}e^{z\cdot X-z\cdot z/2 - \norm{z}^2/2}(X-z)\right\}. 
\end{align}
Simple algebra gives $|e^{z\cdot X-z\cdot z/2 - \norm{z}^2/2}| = e^{a\cdot X-\norm{a}^2}$, so that by Cauchy--Schwarz, 
\begin{align*}
 \norm{D_zT(R_m/r_G)}
 &\le\sqrt{r_\star(2a)}\Expect_0\left\{
   \frac{|R_m(X)|}{r_G(X)}e^{a\cdot X-\norm{a}^2}\norm{X-z}\right\}\\
 &\le \norm{R_m}_{L_2(\varphi)}
 \left[r_\star(2a)\Expect_0\left\{
   \frac{e^{2a\cdot X-2\norm{a}^2}\norm{X-z}^2}{r_G(X)^2}
 \right\}\right]^{1/2}.
\end{align*}
For $\norm{z} = R$, we bound the second factor as follows: 
\begin{align}
 r_\star(2a)\Expect_0\left\{
   \frac{e^{2a\cdot X-2\norm{a}^2}\norm{X-z}^2}{r_G(X)^2}
 \right\}
 &\stepa{=}r_\star(2a)\Expect_0\left\{
   \frac{\norm{X+a}^2+\norm{b}^2}{r_G(X+2a)^2}\right\}\notag\\
 &\stepb{\le} e^{\gamma_K^2+2\gamma_K\norm{a}}
       \Expect_0\{(\norm{X+a}^2+\norm{b}^2)e^{2\gamma_K\norm{X+2a}}\}\notag\\
 &\stepc{\le} 2e^{\gamma_K^2+6\gamma_K\norm{a}}
       \Expect_0\{(\norm{X}^2 + R^2)e^{2\gamma_K\norm{X}}\}\notag\\
 &\stepd{\le} C_K(1+R^2)e^{6\gamma_K\norm{a}}. \label{eq:normalized-kernel-second-moment}
\end{align}
Here (a) uses $\varphi(x)e^{2a\cdot x-2\norm{a}^2}=\varphi(x-2a)$ and a change of variable $X\mapsto X+2a$, (b) follows from the bounds $r_\star(2a)\le e^{2\gamma_K\norm{a}}$ and $r_G(x)^{-2}\le e^{\gamma_K^2+2\gamma_K\norm{x}}$ in \prettyref{eq:r-bound}, (c) uses $\norm{X+a}^2+\norm{b}^2\le 2(\norm{X}^2+R^2)$, and (d) follows from the finiteness of corresponding Gaussian moments. Taking square roots yields
\begin{equation}
 \norm{D_zT(R_m/r_G)}
 \le C_K(1+R)e^{3\gamma_KR}\norm{R_m}_{L_2(\varphi)}.
 \label{eq:tail-shell-bound}
\end{equation}
Using \prettyref{eq:remainder-bound} and the choice $m=\lfloor \kappa R^2 \rfloor$, we arrive at \prettyref{eq:scaled-tail-shell-bound} by absorbing factors into $e^{-cR^2\log R}$. \\ 

\noindent \emph{Contribution of the polynomial.} Next we treat the polynomial part $P_m$ and show that
\begin{equation}
 \sqrt n\norm{D_zT(P_m/r_G)}
 \le Ch e^{-cR} \qquad \text{for all } z\in \mathcal S_R.\label{eq:polynomial-target}
\end{equation}
To prove \prettyref{eq:polynomial-target}, we distinguish two cases according to $a=\Rea z$. 

\paragraph{Case I: $\norm{a}\ge D_K\sqrt{m}$.} Here $D_K>0$ is a large constant depending only on $K$. 
Applying the change of variable $x=y+a$ and $\varphi(y+a)=\varphi(y)e^{-a\cdot y-\norm{a}^2/2}$ to \prettyref{eq:DzT-formula}, we obtain the exact vector identity
\[
 D_zT(P_m/r_G)
 =\sqrt{r_\star(2a)}e^{-\norm{a}^2/2}
 \int_{\R^d}\varphi(y)\frac{P_m(y+a)}{r_G(y+a)}
       e^{ib\cdot y}(y-ib)\d y.
\]
To proceed, \prettyref{eq:r-bound} gives
\begin{align*}
 \frac1{r_G(y+a)}
 \le e^{\gamma_K\norm{y+a}+\gamma_K^2/2}
 \le e^{\gamma_K\norm{a}+\gamma_K^2/2}e^{\gamma_K\norm{y}}, \qquad \sqrt{r_\star(2a)}\le e^{\gamma_K\norm{a}}.
\end{align*}
In addition, $\norm{y-ib}\le \norm{y}+R\le (1+R)(1+\norm{y})$.  Taking the norm inside the integral therefore gives
\begin{align*}
 \norm{D_zT(P_m/r_G)}
 &\le (1+R)e^{-\norm{a}^2/2+2\gamma_K\norm{a}+\gamma_K^2/2}
 \Expect_0\{|P_m(X+a)|e^{\gamma_K\norm{X}}(1+\norm{X})\} \\
 &\le (1+R)e^{-\norm{a}^2/2+2\gamma_K\norm{a}+\gamma_K^2/2}\cdot \norm{P_m(\cdot+a)}_{L_2(\varphi)}
 \left[\Expect_0\{e^{2\gamma_K\norm{X}}(1+\norm{X})^2\}\right]^{1/2} \\
 &\le C_K(1+R)e^{-\norm{a}^2/2+2\gamma_K\norm{a}}
\norm{P_m(\cdot+a)}_{L_2(\varphi)},
\end{align*}
where the second inequality is Cauchy--Schwarz. Applying \prettyref{lem:Gaussian-translation} to $P_m$ with shift $a$ and using \prettyref{eq:polynomial-bound}, we obtain
\begin{equation}
 \norm{D_zT(P_m/r_G)}
 \le C_K(1+R)H
 \exp\left\{-\frac{\norm{a}^2}{2}+\norm{a}\sqrt m+C_K\norm{a}+C_K\sqrt m\right\}.
 \label{eq:large-real}
\end{equation}
Choose $D_K$ large enough that, for $m\ge1$ and $\norm{a}\ge D_K\sqrt m$, the exponent in \prettyref{eq:large-real} is at most $-\norm{a}^2/4$. With $m=\lfloor \kappa R^2\rfloor$, \prettyref{eq:large-real} implies \prettyref{eq:polynomial-target} in this case (with an even better exponent $e^{-cR^2}$). 

\paragraph{Case II: $\norm{a}<D_K\sqrt m$.}
Choose a small enough $\kappa$ so that $D_K\sqrt\kappa\le1/2$.  Then $\norm{a}\le R/2$, so $\norm{b}\ge R/2>0$; let $e=b/\norm{b}$. For the constant $\tau_K>0$ given in \prettyref{lem:zero-free}, we begin with a key contour identity: 
\begin{align}
 D_zT(P_m/r_G)
 &=\sqrt{r_\star(2a)}e^{-\norm{z}^2/2}\int_{\R^d}\varphi(x)\frac{P_m(x)}{r_G(x)}
       e^{z\cdot x-z\cdot z/2}(x-z)\d x  \label{eq:population-contour-shift} \\
 &=\sqrt{r_\star(2a)}e^{-\norm{z}^2/2}\int_{\R^d}\varphi(x+i\tau_Ke)
       \frac{P_m(x+i\tau_Ke)}{r_G(x+i\tau_Ke)}
       e^{z\cdot(x+i\tau_Ke)-z\cdot z/2}
       (x+i\tau_Ke-z)\d x. \notag
\end{align}
Here the first equality is the same as \prettyref{eq:DzT-formula}, so the key is to justify the second equality. Denote the vector integrand in the first line by $\Phi_z(x)$, i.e.,
\begin{align*}
\Phi_z(x) = \varphi(x)\frac{P_m(x)}{r_G(x)}
       e^{z\cdot x-z\cdot z/2}(x-z).
\end{align*}
Decomposing $\R^d$ into the direction $e$ and its orthogonal complement $e^\perp$, by Fubini it suffices to show that, for every $u\in e^\perp$,
\begin{align}\label{eq:contour_identity}
\int_{\R} \Phi_z(u+te) \d t = \int_{\R} \Phi_z(u+(t+i\tau_K)e) \d t.  
\end{align}
Indeed, every $x\in \R^d$ can be written as $x = u + te$ with $u\in e^\perp$, so integrating \prettyref{eq:contour_identity} over $u\in e^\perp$ and applying Fubini give $\int_{\R^d} \Phi_z(x)\d x=\int_{\R^d} \Phi_z(x+i\tau_Ke)\d x$, which is \prettyref{eq:population-contour-shift}.

Consider the rectangular contour with vertices $-T,T,T+i\tau_K,-T+i\tau_K$, with $T\to\infty$.
By Cauchy's integral theorem, to prove (\ref{eq:contour_identity}), it further suffices to show that: 
\begin{enumerate}
    \item the map $\zeta\mapsto \Phi_z(u+\zeta e)$ is holomorphic on a neighborhood of the strip $0\le\Ima\zeta\le\tau_K$; 
    \item the total contribution of the vertical sides of the rectangle vanishes as $T\to\infty$: 
    \[
 \sum_{\sigma\in\{-1,1\}}\int_0^{\tau_K}
       \norm{\Phi_z(u+(\sigma T+iv)e)}\d v \longrightarrow 0, \qquad \text{uniformly in }u\in e^\perp. 
\]
\end{enumerate}
For the first claim, it suffices to check that the denominator $r_G(u+\zeta e)$ is zero-free on the strip $0\le\Ima\zeta\le\tau_K$; since $\norm{\Ima(u+\zeta e)} \le \tau_K$, this follows from \prettyref{lem:zero-free}. For the second claim, the polynomial degree bound and \prettyref{eq:r-bound} give, for $0\le v\le\tau_K$,
\[
 \norm{\Phi_z(u+(t+iv)e)}
 \le C(1+\norm{u}+|t|)^{m+1}
       e^{-(\norm{u}^2+t^2)/2+(\gamma_K+\norm{a})(\norm{u}+|t|)},
\]
where $C$ may depend on the fixed $P_m,z,m,K$. For $t\in \{\pm T\}$ and integrating over $v\in [0,\tau_K]$, the second claim follows. This proves (\ref{eq:contour_identity}) and hence also \prettyref{eq:population-contour-shift}. 

In the second integral of \prettyref{eq:population-contour-shift}, note that
\begin{align*}
 |\varphi(x+i\tau_Ke)|&=\varphi(x)e^{\tau_K^2/2},\\
 e^{-\norm{z}^2/2}\left|e^{z\cdot(x+i\tau_Ke)-z\cdot z/2}\right|
 &=e^{a\cdot x-\norm{a}^2-\tau_K\norm{b}}.
\end{align*}
Taking the norm inside the integral and using \prettyref{lem:zero-free}, we have
\begin{align*}
 \norm{D_zT(P_m/r_G)}
 &\le e^{\tau_K^2/2-\tau_K\norm{b}}\sqrt{r_\star(2a)}
 \Expect_0\left\{\frac{|P_m(X+i\tau_Ke)|}{r_G(X)/2}
 e^{a\cdot X-\norm{a}^2}
 \norm{X+i\tau_Ke-z}\right\}\\
 &\le2e^{\tau_K^2/2-\tau_K\norm{b}}
 \norm{P_m(\cdot+i\tau_Ke)}_{L_2(\varphi)}
 \left[r_\star(2a)\Expect_0\left\{
 \frac{e^{2a\cdot X-2\norm{a}^2}
       \norm{X+i\tau_Ke-z}^2}
      {r_G(X)^2}
 \right\}\right]^{1/2}.
\end{align*}
Similar to the analysis in \prettyref{eq:normalized-kernel-second-moment}, the final factor is at most $C_K(1+R)e^{3\gamma_K\norm{a}}$. Therefore, 
\begin{align}
 \norm{D_zT(P_m/r_G)}
 &\le C_K(1+R)e^{-\tau_K\norm{b}+3\gamma_K\norm{a}}
\norm{P_m(\cdot+i\tau_Ke)}_{L_2(\varphi)} \notag \\
  &\le C_K(1+R)H
       \exp\{-\tau_K\norm{b}+3\gamma_K\norm{a}+C_K\sqrt m\},
 \label{eq:large-imag}  
\end{align}
where the last step uses \prettyref{lem:Gaussian-translation} applied to $P_m$ with shift $i\tau_Ke$, as well as \prettyref{eq:polynomial-bound}. 
Choose $\kappa$, depending only on $K$, small enough that in addition
\[
 (3\gamma_KD_K+C_K)\sqrt\kappa\le\frac{\tau_K}{4}.
\]
Then $\norm{b}\ge R/2$, $\norm{a}<D_K\sqrt m$, and $\sqrt m\le\sqrt\kappa R$ make the exponent in \prettyref{eq:large-imag} at most $-\tau_KR/4$. This proves \prettyref{eq:polynomial-target} in the case where $\norm{a}$ is small. \\

\noindent \emph{Summary.} Summing up, both bounds \prettyref{eq:population-shell} and \prettyref{eq:population-shell-atomic} follow from \prettyref{eq:scaled-tail-shell-bound} and \prettyref{eq:polynomial-target}. 
\end{proof}

\subsection{Empirical approximation error}\label{subsec:fitted_correction_err}

The estimation error term decomposes into its population counterpart and an empirical approximation error:
\begin{equation}\label{eq:empirical-decomp}
 T_nf_{n,G}=Tf_{n,G}+V_{n,G},
 \qquad
 V_{n,G}(z):=(\Expect_n-\Expect_\star)\{f_{n,G}L_z\}.
\end{equation}
Here $V_{n,G}=(T_n-T)f_{n,G}$ is the empirical approximation error, and we recall the definition of $L_z$ in \prettyref{eq:Lz}. The next proposition controls the normalized gradient of $V_{n,\widehat G}$, with the data-dependent NPMLE $\widehat G$, through a uniform upper bound of $\norm{D_zV_{n,G}}$ over all $G\in \calP(K)$ with $\sqrt{n} H(p_G,p_\star)\le h$. 

\begin{proposition}\label{prop:empirical-shell}
There are constants $C,c,R_0,n_0>0$, depending only on $K$, such that the following holds.  For every $n\ge n_0$, $R\ge R_0$, and $h>1$ satisfying the sample-size condition
\[
 \log n\ge C(R^2+\log h),
\]
uniformly over $G_\star\in\cP(K)$, with probability at least $1-n^{-c}$,
\begin{equation}
 \sup_{\substack{G\in\cP(K):\\ \sqrt nH(p_G,p_\star)\le h}}
 \sup_{z\in\mathcal S_R}\norm{D_zV_{n,G}}
 \le n^{-c}.
 \label{eq:empirical-shell}
\end{equation}
\end{proposition}

\begin{proof}

Throughout the proof, let $\eta\in (0,\frac{1}{10})$ be a fixed universal constant; we will use $n^\eta$ to denote a small polynomial factor in $n$. 
Define
\begin{equation} \label{eq:Kz}
 \K_z(x) := D_z L_z = \sqrt{r_\star(2\Rea z)}  \frac{e^{z\cdot x - z\cdot z/2 -\norm{z}^2/2}}{r_\star(x)}(x-z)\in\C^d,
\end{equation}
where we apply $\partial_z L_z(x) = L_z(x)(x-z)$. Under this notation, 
\begin{equation}
 D_zV_{n,G}
 =(\Expect_n-\Expect_\star)\{f_{n,G}\K_z\}.
 \label{eq:empirical-gradient}
\end{equation}
We first bound its variance for fixed $G$ and $z$.  Abbreviate $H\equiv H(p_G,p_\star)$ and, within this calculation, $u=\sqrt{r_G}-\sqrt{r_\star}$, so that $\Expect_0 u^2=H^2$.  By \prettyref{eq:r-bound} and $|u|\le\sqrt{r_G}+\sqrt{r_\star}$,
\[
 \frac{[\Delta r_G(x)]^2}{r_G(x)^2r_\star(x)}
 =\frac{u(x)^2(\sqrt{r_G(x)}+\sqrt{r_\star(x)})^2}
        {r_G(x)^2r_\star(x)}
 \le C_K|u(x)|e^{C_K\norm{x}}.
\]
As a result, with $a=\Rea z$, we have the upper bound
\begin{align}
 \frac1n\Expect_\star\norm{f_{n,G}\K_z}^2
 &=r_\star(2a)\Expect_0\left\{
   \frac{(\Delta r_G)^2}{r_G^2r_\star}
   e^{2a\cdot X-2\norm{a}^2}\norm{X-z}^2\right\}\notag\\
 &\le C_Kr_\star(2a)\left[\Expect_0\left\{u(X)^2
 \right\}\right]^{1/2}
 \left[\Expect_0\left\{
 e^{4a\cdot X-4\norm{a}^2+2C_K\norm{X}}\norm{X-z}^4
 \right\}\right]^{1/2}\notag\\
 &\le C_KH e^{C_K R^2}.
 \label{eq:empirical-variance-raw}
\end{align}
Here the first inequality is Cauchy--Schwarz, and the second inequality uses $r_\star(2a)\le e^{2\gamma_K\norm{a}}\le e^{2\gamma_KR}$ and the Gaussian moment generating function. Since $H\le h/\sqrt n$ and $R^2+\log h\le C^{-1}\log n$, choosing $C$ and then $n_0$ sufficiently large gives
\begin{equation}
 \sup_{G,z}\frac1n\Expect_\star\norm{f_{n,G}\K_z}^2
 \le C_Kh n^{-1/2}e^{C_KR^2}
 \le n^{-1/2+\eta},
 \label{eq:empirical-variance}
\end{equation}
where the supremum is over $G\in\cP(K)$ with $\sqrt nH(p_G,p_\star)\le h$ and $\norm{z}\le R$. Note that using \prettyref{eq:empirical-gradient} and \prettyref{eq:empirical-variance}, Chebyshev's inequality already leads to \prettyref{eq:empirical-shell} for \emph{fixed} $G$ and $z$. 

The rest of the proof is to apply an $\varepsilon$-net argument to establish \prettyref{eq:empirical-shell} uniformly in $G$ and $z$. To begin with, we apply a truncation argument: for a constant $C_{\mathrm{tr},K}$ to be chosen below, define
\begin{equation}
 t_n=C_{\mathrm{tr},K}\sqrt{\log n},\qquad
 \mathcal X_n=\{x\in\R^d:\norm{x}\le t_n\},\qquad
 \mathcal E_n=\bigcap_{i=1}^n\{X_i\in \mathcal X_n\}.
 \label{eq:empirical-truncation-event}
\end{equation}
We establish several properties of $\mathcal X_n$ and $\mathcal E_n$: 
\begin{enumerate}
    \item\label{property1} First, $\sup_{G_\star\in\cP(K)}\Prob_{G_\star}(\mathcal E_n^c) \le n^{-20}$ when $C_{\mathrm{tr},K}$ and then $n_0$ are sufficiently large. This is because each observation has the representation $X_i=\Theta_i+Z_i$, with $\norm{\Theta_i}\le \gamma_K$ and $Z_i\sim N(0,I)$. The usual Gaussian concentration and union bound give this claim. 
    \item\label{property2} On $\calX_n$, one has $|f_{n,G}|\le n^{1/2+\eta}$ for $n_0$ large enough. In fact, by definition of $f_{n,G}$ in \prettyref{eq:fG} and \prettyref{eq:r-bound}, we have $|f_{n,G}|\le C_K\sqrt{n}e^{C_K\sqrt{\log n}}$, which is at most $n^{1/2+\eta}$ for $n\ge n_0$ large enough. The same arguments also give $n^{-\eta}\le r_G\le n^{\eta}$ on $\calX_n$ for all $G\in \calP(K)$. 
    \item\label{property3} On $\calX_n$ and $\{\norm{z}\le R\}$, one has $\norm{\K_z}\le n^\eta$ for $C$ and $n_0$ large enough. In fact, by \prettyref{eq:Kz} and \prettyref{eq:r-bound}, it is clear that $\norm{\K_z} \le C_Ke^{C_K R\sqrt{\log n}}$, which is at most $n^\eta$ for $n\ge \max\{n_0,e^{CR^2}\}$ with $C$ and $n_0$ large enough. 
    \item\label{property4} On $\calX_n$ and $\{\norm{z}\le R\}$, one has $\max_{i\in [2d]}\norm{\partial_{\xi_i} \K_z}\le n^\eta$ for $C$ and $n_0$ large enough, where $\xi\equiv(a,b)\in\R^{2d}$ consists of the real coordinates of $z=a+ib$. This follows from a similar analysis to Property~\ref{property3}, with an additional observation that $\log r_\star(2a)$, which is a log-partition function, has bounded derivatives thanks to the compact support of $G_\star$ (cf. \prettyref{eq:logr-derivative}). In particular, this bound implies that on $\mathcal X_n$, 
    \begin{equation}\label{eq:Kz-Lipschitz}
    \norm{ \K_{z_1} - \K_{z_2} } \le n^\eta \norm{z_1 - z_2}, \qquad \text{for all }z_1,z_2\in \mathcal S_R,
    \end{equation}
    by slightly enlarging $\eta$ if necessary to absorb additional constant factors. 
    \item\label{property5} For large enough $C$ and $n_0$, it holds that \begin{align}
 \sup_{z\in \mathcal S_R}
 \Expect_\star\left[\norm{f_{n,G}\K_z}
       \mathbf1_{\mathcal X_n^c}\right]
 \le n^{-20}.
 \label{eq:empirical-truncation-tail}
\end{align}
This is because the formulas for $f_{n,G}$ and $\K_z$, together with \prettyref{eq:r-bound}, give
\[
\norm{f_{n,G}(x)\K_z(x)}
 \le C_K\sqrt n\,(1+R+\norm{x})^2
       e^{C_KR^2+(2R+C_K)\norm{x}},
 \qquad z\in \mathcal S_R. 
\]
Integrating $\norm{f_{n,G}(x)\K_z(x)} \varphi(x)r_\star(x)$ on the event $\{\norm{x}>t_n\}$ gives a factor of $C_K\sqrt{n} e^{C_KR^2-c_Kt_n^2}$, as long as the constant $C_{\mathrm{tr},K}$ is large enough. Then \prettyref{eq:empirical-truncation-tail} follows from $R^2\le C^{-1}\log n$ for suitably large $C$.
\end{enumerate}

Now we are ready to present detailed estimates on the truncated sum for fixed $z,G$
and an $\varepsilon$-net argument to obtain uniform bound in both arguments. We define the truncated empirical approximation error by
\[
 V_{n,G}^{\mathrm{tr}}(z)
 :=(\Expect_n-\Expect_\star)\{f_{n,G}L_z\mathbf1_{\mathcal X_n}\}.
\]
Then $D_zV_{n,G}^{\mathrm{tr}}=(\Expect_n-\Expect_\star)\{f_{n,G}\K_z\mathbf1_{\mathcal X_n}\}$.
Recall Bernstein's inequality for a real bounded variable $Y$:
\[
 \Prob\{|(\Expect_n-\Expect_\star)Y|>t\}
 \le2\exp\left\{-\frac{cnt^2}
 {\Expect_\star Y^2+\norm{Y}_{\infty} t}\right\}.
\]
Applying this to the real and imaginary parts of each of the $d$ coordinates, together with the variance bound $\E_\star Y^2\le n^{1/2+\eta}$ in \prettyref{eq:empirical-variance} and $\norm{Y}_{\infty} \le n^{1/2+2\eta}$ by Properties~\ref{property2}--\ref{property3}, we obtain 
\begin{equation}
 \Prob\left\{
\norm{D_zV_{n,G}^{\mathrm{tr}}}>n^{-\eta}
 \right\}
 \le C_K\exp\{-c_K n^{1/2-3\eta}\}
 \label{eq:Bernstein-empirical}
\end{equation}
for every fixed $G$ in the Hellinger neighborhood and $z\in\mathcal S_R$. In addition, on the event $\mathcal E_n$, the empirical average of the discarded part is zero, so
\[
 D_zV_{n,G}-D_zV_{n,G}^{\mathrm{tr}}
 =-\Expect_\star\{f_{n,G}\K_z\mathbf1_{\mathcal X_n^c}\}.
\]
Consequently, \prettyref{eq:empirical-truncation-tail} in Property~\ref{property5} gives, on this event,
\begin{equation}\label{eq:truncation-difference}
 \sup_{G\in\cP(K)} \sup_{z\in \mathcal S_R}\norm{D_zV_{n,G}
                  -D_zV_{n,G}^{\mathrm{tr}}}
 \le n^{-20}.
\end{equation}

Next we construct $\varepsilon$-nets over both $G$ and $z$. The $\varepsilon$-net over $z\in \mathcal S_R$ is easy: we take a $\delta_n$-net $\mathcal Z_n=\{z_1,\ldots,z_{N_z}\}\subseteq\mathcal S_R$ in Euclidean distance, with $\delta_n=n^{-12}$.  Since $\mathcal S_R$ is a $(2d-1)$-dimensional sphere, we may take $N_z\le C(1+R/\delta_n)^{2d-1}$, so $\log N_z=O(\log n)$. 

To construct an $\varepsilon$-net over $G$, let
\[
 \mathcal H_n=\{G\in\cP(K):\sqrt nH(p_G,p_\star)\le h\},
 \qquad
 \norm{f}_{\infty,t_n}=\sup_{\norm{x}\le t_n}|f(x)|.
\]
Let $\{r_{G_1},\ldots,r_{G_N}\}$ be an $\eps_n$-net for the density ratios
$\{r_G:G\in\mathcal H_n\}$ in the truncated sup-norm $\norm{\cdot}_{\infty,t_n}$, where $\eps_n=n^{-15}$ and every $G_j$ belongs to $\mathcal H_n$. Thus, for every $G\in\mathcal H_n$, some $j$ satisfies $\| r_G-r_{G_j}\|_{\infty,t_n}\le\eps_n$. 
Using existing covering number bound for mixture density under the same norm, we can ensure $\log N \le C_K (\log n)^{d+1}$. 
Indeed, set
\[
 \gamma_n=\frac{\eps_n}{2}\inf_{x\in\mathcal X_n}\varphi(x)
 =\frac{n^{-15}}{2}(2\pi)^{-d/2}e^{-t_n^2/2}.
\]
By \cite[Theorem~4.1, (4.1)]{SahaGuntuboyina2020} (see earlier \cite[Lemma 2]{Zhang2009} for one dimension), the class $\sth{p_G: G\in \calP(K)}$ admits a $\gamma_n$-cover under $\norm{\cdot}_{\infty,t_n}$ with log-cardinality at most $C_K(\log n)^{d+1}$.  
From each covering ball intersecting $\{p_G:G\in\mathcal H_n\}$, choose a representative $p_{G_j}$ with $G_j\in\mathcal H_n$.  The resulting net has $\log N\le C_K(\log n)^{d+1}$, and every $G\in\mathcal H_n$ is covered by some $G_j$, so that
\[
 \norm{r_G-r_{G_j}}_{\infty,t_n}
 \le\frac{\| p_G-p_{G_j}\|_{\infty,t_n}}
              {\inf_{x\in\mathcal X_n}\varphi(x)}
 \le\frac{2\gamma_n}{\inf_{x\in\mathcal X_n}\varphi(x)}
 =\eps_n.
\]

We now carry out the approximation arguments and bound the error from replacing $G,z$ by net elements $G_j,z_k$.  Choose $j,k$ so that $\|r_G-r_{G_j}\|_{\infty,t_n}\le\eps_n$ and $\norm{z-z_k}\le\delta_n$. By Property~\ref{property2} and the $G$-covering, 
\[
 |f_{n,G}(x)-f_{n,G_j}(x)|
 =\sqrt n\,\frac{r_\star(x)|r_G(x)-r_{G_j}(x)|}
 {r_G(x)r_{G_j}(x)}
 \le n^{-14+3\eta}.
\]
As a result, by Properties~\ref{property2}--\ref{property4}, 
\begin{align*}
 \sup_{x\in\mathcal X_n}
 \norm{f_{n,G}\K_z-f_{n,G_j}\K_{z_k}}
 &\le
 \sup_{x\in\mathcal X_n}|f_{n,G}-f_{n,G_j}|\norm{\K_z}
 +\sup_{x\in\mathcal X_n}|f_{n,G_j}|\norm{\K_z-\K_{z_k}}\\
 &\le n^{-14+4\eta}
           +n^{1/2+2\eta}\delta_n
 \le2n^{-10}.
\end{align*}
Since $\norm{(\Expect_n-\Expect_\star)h}\le2\sup_x\norm{h(x)}$ for any bounded vector-valued function $h$, this estimate gives, for sufficiently large $n$,
\[
 \norm{D_zV_{n,G}^{\mathrm{tr}}
       -D_{z_k}V_{n,G_j}^{\mathrm{tr}}}
 \le4n^{-10}\le n^{-9}.
\]
Taking suprema yields the deterministic reduction
\begin{equation*}
 \sup_{G\in\mathcal H_n}\sup_{z\in\mathcal S_R}
 \norm{D_zV_{n,G}^{\mathrm{tr}}}
 \le\max_{1\le j\le N}\max_{1\le k\le N_z}
 \norm{D_{z_k}V_{n,G_j}^{\mathrm{tr}}}+n^{-9}.
\end{equation*}
On the event $\mathcal E_n$ from \prettyref{eq:empirical-truncation-event}, the inequality \prettyref{eq:truncation-difference} also gives
\begin{equation}
 \sup_{G\in\mathcal H_n}\sup_{z\in\mathcal S_R}
 \norm{D_zV_{n,G}}
 \le\max_{1\le j\le N}\max_{1\le k\le N_z}
 \norm{D_{z_k}V_{n,G_j}^{\mathrm{tr}}}+2n^{-9}.
 \label{eq:empirical-net-reduction-full}
\end{equation}

It remains to bound this finite maximum. Using \prettyref{eq:Bernstein-empirical} and the union bound, we have
\[
 \Prob\left\{
 \max_{1\le j\le N}\max_{1\le k\le N_z}
 \norm{D_{z_k}V_{n,G_j}^{\mathrm{tr}}}>n^{-\eta}
 \right\}
 \le C_K NN_z\exp\{-c_Kn^{1/2-3\eta}\}.
\]
Combining this with \prettyref{eq:empirical-net-reduction-full} and $\Prob(\mathcal E_n^c)\le n^{-20}$ from Property~\ref{property1}, we obtain
\[
 \Prob\left\{
 \sup_{G\in\mathcal H_n}\sup_{z\in\mathcal S_R}
 \norm{D_zV_{n,G}}>n^{-\eta} + 2n^{-9}
 \right\}
 \le n^{-20}+C_K NN_z\exp\{-c_1n^{1/2-3\eta}\}
 \le n^{-\eta}.
\]
The desired inequality \prettyref{eq:empirical-shell} then follows
from $\log(NN_z)=O((\log n)^{d+1})$ and $1/2-3\eta>0$.  All constants are uniform over $G_\star\in\cP(K)$.
\end{proof}

\subsection{Proof of \prettyref{prop:fitted-shell}}
We are now ready to conclude the proof of \prettyref{prop:fitted-shell} from \prettyref{lem:population-shell} and \prettyref{prop:empirical-shell}. By the triangle inequality, 
\begin{align*}
\norm{D_z T_n f_{n,\widehat G}} \le \norm{D_z Tf_{n,\widehat G}} + \norm{D_z V_{n,\widehat G}}. 
\end{align*}
Let $\mathcal E_h$ be the event that $\sqrt nH(p_{\widehat G},p_\star)\le h$; by definition, $\Prob_{G_\star}(\mathcal E_h^c)=a_n(h)$. \prettyref{lem:population-shell} shows that the event $\mathcal E_h$ implies
\begin{align*}
\sup_{z\in \mathcal S_R} \norm{D_zTf_{n,\widehat G}} \le \begin{cases}
Che^{-cR} + C\sqrt{n}e^{-cR^2\log R} &\text{if } G_\star\in \calP(K), \\
Che^{-cR}  &\text{if } G_\star\in \calP_s(K). 
\end{cases}
\end{align*}
For the empirical approximation error, \prettyref{prop:empirical-shell} yields
\begin{align*}
\Prob\pth{ \calE_h \cap \sth{
 \sup_{z\in\mathcal S_R}
 \norm{D_zV_{n,\widehat G}} > n^{-c}} } \le n^{-c}. 
\end{align*}
Therefore, with a common failure probability $a_n(h)+n^{-c}$, the claimed upper bounds in \prettyref{prop:fitted-shell} hold. 

\section{Gradient lower bound of the empirical fluctuation}
\label{sec:centered-gradient}

Recall the normalized gradient defined in \prettyref{eq:gradient-field} and the fluctuation term \prettyref{eq:Wn}.
We now bound the probability that $\|D_zW_n\|$ is small 
somewhere on the sphere $\calS_R$. 
This is the first term on the right-hand side of \prettyref{eq:gradient-triangle}.
This argument applies to every $G_\star$ supported on the fixed compact set $K$. 

\begin{proposition}
\label{prop:centered-gradient-lower}
There are $R_0,C>0$, depending only on $K$, such that the following
holds.  Define
\begin{equation}
 \Delta_{n,R}=Cn^{-1/6}e^{CR^2}.
 \label{eq:point-resolution}
\end{equation}
For every $R\ge R_0$ and $0<\eps<1/4$ satisfying
$\eps\ge\Delta_{n,R}$,
\begin{equation}\label{eq:direct-sphere-final}
\Prob\Big(\inf_{z\in\mathcal S_R}\norm{D_zW_n}\le\eps\Big)
\le CR^{4d+8}\eps^{1/(d+1)}.
\end{equation}
\end{proposition}

By \prettyref{eq:Wn} and the fact that $\E_\star \K_z = 0$, we have
\begin{equation}\label{eq:normalized-field}
D_zW_n = \sqrt n(\Expect_n-\Expect_\star)\K_z = \frac1{\sqrt n}\sum_{i=1}^n\K_z(X_i),
\end{equation}
where $\K_z$ is given by (\ref{eq:Kz}). This exact centering will be used in the covariance and derivative moment estimates.

\subsection{Discretization and the union bound}

We reduce the continuous infimum over the sphere
to a finite minimum using a Lipschitz bound.

\begin{lemma}\label{lem:covering_lemma}
Let $R\ge1$, $0<\eps<1$, and $L\ge1$. Let $\Lambda_R$ denote the Lipschitz constant of $z\mapsto D_zW_n$ on $\calS_R$ with respect to the Euclidean distance. Set $\delta=\eps/L$, and choose a deterministic maximal $\delta$-separated subset
$\N_\delta\subset\calS_R$. Then $\N_\delta$ is a $\delta$-net with
\begin{equation}
|\N_\delta|\le C_d(R/\delta)^{2d-1},
\label{eq:Ndelta}
\end{equation}
and
\begin{equation}\label{eq:direct-sphere-union}
\Prob\Big(\inf_{z\in\calS_R}\norm{D_zW_n}\le\eps\Big)
\le \sum_{z\in\N_\delta}\Prob\big(\norm{D_zW_n}\le2\eps\big)
+ \Prob\big(\Lambda_R>L\big).
\end{equation}
\end{lemma}

\begin{proof}
Maximality implies that $\N_\delta$ is a
$\delta$-covering (see, e.g., \cite[Theorem 27.2]{PW-it}). The cardinality bound follows
from the standard packing bound for the
$(2d-1)$-dimensional sphere; see
\cite[Corollary 27.4]{PW-it}.

On the event $\{\Lambda_R\le L\}$, for every
$w\in\calS_R$ there exists $z\in\N_\delta$
such that $\norm{w-z}\le\delta$. Hence
\[
\norm{D_wW_n}\ge \norm{D_zW_n}-L\norm{w-z} \ge \norm{D_zW_n}-L\delta.
\]
Taking the infimum over $w$ gives
\[
\inf_{w\in\calS_R}\norm{D_wW_n} \ge\min_{z\in\N_\delta}\norm{D_zW_n} - \eps.
\]
The claim follows from the union bound.
\end{proof}

To apply \prettyref{eq:direct-sphere-union}, we next bound the pointwise small-ball probabilities $\Prob(\norm{D_zW_n}\le\eps)$ in
\prettyref{sec:small-ball} and control $\Lambda_R$ in \prettyref{sec:lip}. We then complete the proof of
\prettyref{prop:centered-gradient-lower} in \prettyref{sec:pf-centered-gradient-lower}.

\subsection{Anti-concentration  estimates}
\label{sec:small-ball}

For $z\in\C^d$, let $\tau_z=1\wedge\norm{\Ima z}$.  The following estimate controls the pointwise probabilities in \prettyref{eq:direct-sphere-union}.

\begin{lemma}\label{lem:point-small-ball}
There exists $C>0$ depending only on $K$ (not on $G_\star$), such that the following holds for $R \geq 1$ and $\norm{z}\leq 3R$. Let $\Delta_{n,R}$ be as in
\prettyref{eq:point-resolution}, with its constant chosen sufficiently large.
For $\Delta_{n,R}\le\eps\le1/2$,
\begin{equation}\label{eq:point-small-ball}
\Prob\big(\norm{D_zW_n} \leq \eps\big) \leq C(1+R)^{2d}\frac{\eps^{2d}}{(\eps^2 + \tau_z^2)^{d/2}}.
\end{equation}
\end{lemma}

The proof applies small-ball probability estimate for the empirical average $D_zW_n$ in \prettyref{eq:normalized-field}
based on its second and third moment, which we first control.

\subsubsection{Covariance and third moment}

For $z\in\C^d$ define the real-valued weight
\begin{equation} \label{eq:second-moment-weight}
\rho_z(t) = \varphi(t)\frac{r_\star(2\Rea z)}{r_\star(2\Rea z+t)}, \quad t\in\R^d,
\end{equation}
where $\varphi$ is the standard Gaussian density in $\R^d$, and $r_\star = r_{G_\star}$ is defined by (\ref{eq:rG}). This weight determines the covariance of $\K_z$ after the change of variables $t=x-2\Rea z$, as shown by (\ref{eq:cov_second_moment}) below.

\begin{lemma}\label{lem:weight-bounds}
\begin{enumerate}
\item Define
\begin{equation}\label{eq:gradient-U}
U_a(x) = \sqrt{r_\star(2a)}\frac{e^{a\cdot x-\norm{a}^2}}{r_\star(x)},
\end{equation}
Then with $z = a + bi$, we have
\begin{align*}
\K_z(x) &= U_a(x)e^{ib\cdot (x-a)}(x-z).
\end{align*}
\item For every $x\in \R^{d}$, 
\begin{align}\label{eq:change_measure}
U_a(x)^2p_\star(x) = \rho_z(x - 2a).
\end{align}
\item There exist some $c, C > 0$ depending only on $(M,d)$ such that for any $t\in \R^d$ and $z \in \C^d$,
\begin{align*}
c \varphi_{\frac{1}{2}I}(t) &\leq \rho_z(t) \leq e^{C \norm{t}}\varphi(t).
\end{align*}
Consequently, for every fixed $m\ge 0$,
\begin{equation}
\sup_{z\in\C^d} \int_{\R^d}(1+\norm{t})^m\rho_z(t)\d t \le C_{M,m,d}.
\label{eq:rho-moment}
\end{equation}
\item For every multi-index $\alpha$ with $|\alpha| = m\ge 1$,
\begin{equation}
\sup_{x\in\R^d} |\partial^\alpha\log r_\star(x)| \le C_m M^m.
\label{eq:logr-derivative}
\end{equation}
\end{enumerate}
\end{lemma}

\begin{proof}
(1) This follows directly from (\ref{eq:Kz}).

\noindent(2) For $t=x-2a$, we have
\begin{align*}
U_a(x)^2p_\star(x) = \frac{r_\star(2a)}{r_\star^2(x)} e^{2a\cdot x - 2\norm{a}^2} r_\star(x)\varphi(x) = \frac{r_\star(2a)}{r_\star(x)}\varphi(x-2a) = \frac{r_\star(2a)}{r_\star(2a+t)}\varphi(t) = \rho_z(t), 
\end{align*}
where we use the fact that $\varphi(x) \exp(2a\cdot x - 2\norm{a}^2) = \varphi(x-2a)$.

\noindent (3)
For $x\in\R^d$, define the tilted probability measure
\[
\d G_x(\theta)
=\frac{e^{x\cdot\theta-\norm{\theta}^2/2}}
{r_\star(x)}\,\d G_\star(\theta) = \frac{\varphi(x-\theta)}
{p_\star(x)}\,\d G_\star(\theta).
\]
In other words, $G_x$ is the \emph{posterior} distribution of $\theta \sim G_\star$ given $X=\theta+Z=x$.
Then, for every $t\in\R^d$,
\begin{equation}\label{eq:tilted-weight}
\frac{r_\star(x+t)}{r_\star(x)}
=\int_K e^{\theta\cdot t}\,\d G_x(\theta).
\end{equation}
Since $G_x$ is supported on $K=[-M,M]^d$,
we have $\norm{\theta}\le\sqrt d\,M$. Hence
\[
e^{-\sqrt{d} M\norm{t}} \le \frac{r_\star(x+t)}{r_\star(x)} \le e^{\sqrt{d} M\norm{t}}.
\]
Taking $x=2a$ and recalling the definition
of $\rho_z$, we obtain
\begin{align}\label{eq:rho_z}
\varphi(t)e^{-\sqrt{d} M\norm{t}} \le \rho_z(t) \le \varphi(t)e^{\sqrt{d} M\norm{t}}.
\end{align}
Moreover,
\begin{equation}\label{eq:weight-gaussian-lower}
\rho_z(t) \ge 2^{-d/2}e^{-dM^2/2} \varphi_{\frac12 I}(t),
\end{equation}
where we used $\norm{t}^2/2-\sqrt{d} M\norm{t}\ge-dM^2/2$. The upper bound of $\rho_z(t)$ in \prettyref{eq:rho_z} immediately yields the moment bound \prettyref{eq:rho-moment}.

\noindent (4)
For the derivative bound, taking logarithms in \prettyref{eq:tilted-weight} gives
\[
\log r_\star(x+t)-\log r_\star(x)
=\log\int_K e^{\theta\cdot t}\,\d G_x(\theta).
\]
Thus, for every multi-index $\alpha$ with $|\alpha|=m\ge1$,
$\partial^\alpha\log r_\star(x)$ is an $m$-th order joint cumulant under the posterior $G_x$. By the moment--cumulant formula, it is a finite linear combination of products of moments of total degree $m$. Since $|\theta_j|\le M$ on $K$, each such product is bounded by $M^m$, which proves \prettyref{eq:logr-derivative}.
\end{proof}

Define the real random vector and its covariance by
\begin{equation}\label{eq:gradient-covariance}
Y_z = \begin{pmatrix}\Rea\K_z\\\Ima\K_z\end{pmatrix} \in\R^{2d}, \qquad
\Sigma_z = \Cov(Y_z) = \Expect_\star\{Y_zY_z^\top\},
\end{equation}
where no centering is needed because $\Expect_\star\K_z=0$. By
\prettyref{eq:normalized-field}, $\Sigma_z$ is also the covariance of the $2d$ real coordinates of $D_zW_n$, for every $n$.

\begin{lemma}\label{lem:point-covariance}
There are constants $c,C>0$, depending only on $K$ (and not on $G_\star$), such that the following holds for $R\ge1$.
For any $\norm{z} \leq 3R$, with
\begin{align*}
P_z := 
\begin{pmatrix}
\cos \phi I_d & \sin \phi I_d \\
-\sin \phi I_d & \cos \phi I_d
\end{pmatrix},\quad \phi = \Ima(z) \cdot \Rea(z),
\end{align*}
we have 
\begin{equation}\label{eq:gradient-covariance-bounds}
\Cov
\begin{pmatrix}
\Rea(e^{-i\phi} \K_z) \\
\Ima(e^{-i\phi} \K_z)
\end{pmatrix}
= P_z \Sigma_z P_z^\top \succeq c(1+R)^{-2} \diag(I_d,\tau_z^2I_d),
\end{equation}
where $\tau_z = 1 \wedge \norm{\Ima z}$. 
Moreover,
\begin{equation}\label{eq:third-moment}
m_{3,z}\equiv \Expect_\star\norm{Y_z}^3\le e^{CR^2}.
\end{equation}
\end{lemma}

\begin{proof}
We prove the covariance lower bound first for
$b=\Ima z\ne0$. We start by specifying the rotation used for the $2d$ real coordinates. Write $z=a+ib$, let $r=\norm{b}>0$, and choose
$Q\in O(d)$ such that $Qb=re_1$.
We apply the same rotation $Q$ to $a$, $b$, and
$t=x-2a$, and use the same notation for the
rotated vectors. In particular,
\[
 b=re_1,\quad t=(s,t_\perp),\quad a=(a_1,a_\perp).
\]
The same rotation is applied to the complex vector
$\K_z$. We also remove the constant phase
$e^{ia\cdot b}$. More precisely, define
\[
 \widetilde{\K}_z = e^{-ia\cdot b}Q\K_z.
\]
These operations correspond to the orthogonal transformation
\begin{align}\label{eq:def_Oz}
 O_z=
 \begin{pmatrix}
  \cos\phi\,Q & \sin\phi\,Q\\
  -\sin\phi\,Q & \cos\phi\,Q
 \end{pmatrix}\in O(2d),
 \qquad \phi=a\cdot b,
\end{align}
on the real vector $Y_z$. Hence the covariance
$\widetilde\Sigma_z$ of $(\Rea\widetilde{\K}_z, \Ima\widetilde{\K}_z)$ satisfies $\widetilde\Sigma_z = O_z\Sigma_zO_z^\top$. It suffices to prove
\begin{align}\label{eq:O_lower_bound}
O_z\Sigma_z O_z^\top \succeq c(1+R)^{-2} \diag(I_d,\tau_z^2I_d). 
\end{align}
Indeed, (\ref{eq:gradient-covariance-bounds}) follows from (\ref{eq:O_lower_bound}), since $O_z = \diag(Q,Q)P_z$, and
\begin{align*}
\diag(Q,Q)^\top \cdot c(1+R)^{-2} \diag(I_d,\tau_z^2I_d) \cdot \diag(Q,Q) = c(1+R)^{-2} \diag(I_d,\tau_z^2I_d).
\end{align*}

By \prettyref{eq:gradient-U} and the identity
$x-z=t+\overline z$, we have, in the rotated
coordinates,
\[
 \widetilde{\K}_z(x) = e^{-ia\cdot b}U_a(x)e^{ib\cdot(x-a)} Q(x-z) = U_a(x)e^{irs}
 \begin{pmatrix}
  s+a_1-ir\\
  t_\perp + a_\perp
 \end{pmatrix},
\]
Let 
\begin{align*}
f(s) &:= \Rea\{(s+a_1-ir)e^{irs}\},\\
g(s) &:= \Ima\{(s+a_1-ir)e^{irs}\}.
\end{align*}
Then we have
\begin{align*}
\Rea \widetilde \K_z &= U_a(x)
\begin{pmatrix}
f(s)\\
\cos (rs) \cdot (t_\perp + a_\perp)
\end{pmatrix},\\
\Ima \widetilde\K_z &= U_a(x)
\begin{pmatrix}
g(s)\\
\sin(rs) \cdot (t_\perp + a_\perp)
\end{pmatrix}.
\end{align*}
This further entails that for any $u,v\in\R^d$ (here and below, when $d=1$, the perpendicular vectors and the corresponding sums are empty),
\begin{align*}
u\cdot \Rea \widetilde \K_z + v \cdot \Ima \widetilde \K_z &= U_a(x)h(t),
\end{align*}
where
\begin{align*}
h(t) &= u_1 f(s) + v_1g(s) + \sum_{j=2}^d (a_j + t_j)k_j(s),
\end{align*}
and $k_j(s) = u_j\cos(rs) + v_j\sin(rs)$. Hence, with $\E_0$ denoting the expectation under $N(0,I_d/2)$,
\begin{align}
\E_\star \Big(u\cdot \Rea \widetilde \K_z + v \cdot \Ima \widetilde \K_z\Big)^2
&= \int U_a(2a+t)^2 h(t)^2p_\star(2a+t) \d t\notag\\
&\overset{(i)}{=} \int h(t)^2\rho_z(t)\d t\notag\\
&\overset{(ii)}{\geq} c\, \E_0[h(t)^2],
\label{eq:cov_second_moment}
\end{align}
where $(i)$ is by the change of variable in (\ref{eq:change_measure}), and $(ii)$ is by the Gaussian lower bound \prettyref{eq:weight-gaussian-lower}. Here $\rho_z$ is understood in the rotated coordinates, and the Gaussian lower bound \prettyref{eq:weight-gaussian-lower} remains valid because the rotated mixing measure is supported on $QK$, whose Euclidean radius is still $\gamma_K=\sqrt d\,M$, and the reference Gaussian distribution $N(0,I_d/2)$ is rotationally invariant.

Let us now lower bound this expectation. We start by noting that 
\begin{align}\label{eq:h_second_moment}
\E_0[h(t)^2] &= \E_0 \Big(u_1 f(s) + v_1g(s) + \sum_{j=2}^d a_j k_j(s)\Big)^2 + \frac{1}{2}\sum_{j=2}^d \E_0\big[k_j^2(s)\big].
\end{align}
Recall that $\tau = 1 \wedge \norm{b} = 1 \wedge r$. Next we claim that
\begin{align}
\E_0 \big(u_1 f(s) + v_1 g(s)\big)^2 &\geq c(u_1^2 + v_1^2\tau^2), \label{eq:cov_claim_1}\\
\E_0[k_j^2(s)] = \E_0 \big(u_j \cos(rs) + v_j\sin(rs)\big)^2 & \geq c(u_j^2 + v_j^2\tau^2)\label{eq:cov_claim_2}.
\end{align}
Suppose that \prettyref{eq:cov_claim_1} and \prettyref{eq:cov_claim_2} hold, then (\ref{eq:h_second_moment}) implies
\begin{align*}
\E_0[h(t)^2] &\geq \frac{1}{2}\sum_{j=2}^d \E_0\big[k_j^2(s)\big] \geq c \sum_{j=2}^d (u_j^2 + \tau^2 v_j^2),
\end{align*}
and moreover, 
\begin{align*}
c(u_1^2 + v_1^2\tau^2) &\leq \E_0 \big(u_1 f(s) + v_1 g(s)\big)^2\\
&\leq 2 \E_0\Big(u_1 f(s) + v_1 g(s) + \sum_{j=2}^d a_jk_j(s)\Big)^2 + 2 \E_0 \Big(\sum_{j=2}^d a_jk_j(s)\Big)^2\\
&\leq 2\E_0[h(t)^2] + 2\norm{a_\perp}^2\E_0 \Big[\sum_{j=2}^d k_j^2(s)\Big] \leq 8 (1 + \norm{a_\perp}^2)\E_0[h(t)^2].
\end{align*}
Together they imply that
\begin{align*}
\E_0[h(t)^2] &\geq \frac{c}{1 + \norm{a_\perp}^2}(\norm{u}^2 + \tau^2\norm{v}^2),
\end{align*}
which, in view of (\ref{eq:cov_second_moment}), concludes the covariance lower bound when $b \neq 0$.

It remains to prove \prettyref{eq:cov_claim_1} and \prettyref{eq:cov_claim_2}. For \prettyref{eq:cov_claim_1}, direct calculation shows that $\E_0[fg] = 0$, and
\begin{align*}
\E_0f^2 &= \frac{(a_1^2+\frac12)(1+e^{-r^2})+r^2}{2} \geq c(1+a_1^2),\\
\E_0g^2 &= \frac{(a_1^2+\frac12)(1-e^{-r^2})+r^2}{2}
 \ge c\tau^2(1+a_1^2),
\end{align*}
where we used $1-e^{-r^2}\geq c(1\wedge r^2) = c\tau^2$.
Consequently, for all $u_1,v_1\in\R$,
\begin{equation} \label{eq:gradient-first-gram}
\E_0(u_1f+v_1g)^2 \ge c(1+a_1^2)(u_1^2+\tau^2v_1^2),
\end{equation}
which proves \prettyref{eq:cov_claim_1}. For \prettyref{eq:cov_claim_2}, similar arguments yield
\begin{equation}\label{eq:gradient-perp-gram}
\E_0[k_j^2] = \frac{1+e^{-r^2}}2u_j^2 + \frac{1-e^{-r^2}}2v_j^2 \ge c(u_j^2+\tau^2v_j^2).
\end{equation}

At $b=0$, the kernel $\K_z(x)=U_a(x)(t+a)$ is real.
A similar argument to \prettyref{eq:cov_second_moment} gives, for every $u\in\R^d$,
\[
 \E_\star[(u\cdot\K_z)^2]
 \ge c\,\E_0\big(u\cdot(t+a)\big)^2
 =c\Big(\frac12\norm{u}^2+(u\cdot a)^2\Big).
\]
Since $\K_z$ is real at these points, the imaginary coordinates vanish. The preceding inequality proves the covariance lower bound when $\tau_z=0$.

\medskip
\noindent\textbf{Bound of third moment.}
Write $z=a+ib$. By \prettyref{eq:gradient-U},
\[
\norm{\K_z(x)}=U_a(x)\norm{x-z}.
\]
Let $\gamma_K=\sqrt d\,M$. Compact support of
$G_\star$ gives
\[
r_\star(2a)\le e^{2\gamma_K\norm{a}},
\qquad
r_\star(x)\ge e^{-\gamma_K\norm{x}-\gamma_K^2/2}.
\]
Consequently, for $\norm{z}\le3R$ and $R\ge1$,
\begin{align*}
U_a(x) \le \exp\Big(-\norm{a}^2+\gamma_K\norm{a}+\frac{\gamma_K^2}{2} +(\norm{a}+\gamma_K)\norm{x} \Big) \le C_K e^{C_KR\norm{x}},
\end{align*}
where we used
$-\norm{a}^2+\gamma_K\norm{a}\le \gamma_K^2/4$.
Since $\norm{x-z}\le \norm{x}+3R$, we obtain
\[
\norm{\K_z(x)}
\le C_K(\norm{x}+3R)e^{C_KR\norm{x}}.
\]
Under $p_\star$, write $X=\theta+G$, where
$\theta\sim G_\star$ and $G\sim N(0,I_d)$
are independent. Since $\norm{\theta}\le \gamma_K$, using the 
Gaussian moment generating function, we get
\begin{align*}
m_{3,z} =\E_\star\norm{\K_z(X)}^3 \le C_K\E\left[
(\norm{X}+3R)^3e^{3C_KR\norm{X}}
\right] \le e^{C_KR^2},
\end{align*}
which proves \prettyref{eq:third-moment}.
\end{proof}

\subsubsection{Small ball probability of empirical sum}

For a random vector $X\in\R^k$ and $\eps\ge0$, define its
Euclidean concentration function by
\begin{equation}\label{eq:ac-concentration-function}
 Q(X,\eps) = \sup_{u\in\R^k}\Prob\big(\norm{X-u}\leq \eps\big).
\end{equation}
The following bound on the concentration function is a corollary of multivariate Esseen inequality \cite[Lemma 6.1]{esseen1968concentration}. 
For completeness, we give a proof in \prettyref{app:esseen}.

\begin{lemma}\label{lmm:esseen}
Let $Y_1,\ldots,Y_n$ be iid centered random vectors in $\R^k$, with covariance matrix $\Sigma\succeq0$, and let
$S_n=n^{-1/2}\sum_{i=1}^nY_i$. Suppose that
$m_3=\Expect\norm{Y_1}^3<\infty$. There is a constant $C$, depending only on $k$, such that for every $\eps>0$ satisfying $\sqrt n\eps^3\ge m_3$,
\begin{equation} \label{eq:ac-resolved}
Q(S_n,\eps) \leq \frac{C\eps^k} {\sqrt{\det(\Sigma+\eps^2I)}}.
\end{equation}
\end{lemma}

\begin{proof}[Proof of \prettyref{lem:point-small-ball}]
Apply \prettyref{lmm:esseen} in dimension $2d$ to the iid centered vectors
$Y_z(X_i)$ from \prettyref{eq:gradient-covariance}.  Their normalized sum
consists of the real and imaginary parts of $D_zW_n$, with the same
Euclidean norm.  By \prettyref{eq:third-moment}, $m_{3,z}\le e^{CR^2}$.
Choosing the constant in \prettyref{eq:point-resolution} sufficiently large
ensures $\sqrt n\,\eps^3\ge m_{3,z}$ whenever
$\eps\ge\Delta_{n,R}$.  Thus \prettyref{eq:ac-resolved} gives
\[
\Prob\big(\norm{D_zW_n}\le\eps\big) \le\frac{C\eps^{2d}}
          {\sqrt{\det(\Sigma_z+\eps^2I)}}.
\]
Applying the covariance lower bound of Lemma \ref{lem:point-covariance}, we get
\[
 \sqrt{\det(\Sigma_z+\eps^2I)}
 \ge \{c(1+R)^{-2}+\eps^2\}^{d/2}
      \{c(1+R)^{-2}\tau_z^2+\eps^2\}^{d/2}
 \ge c^d(1+R)^{-2d}(\tau_z^2+\eps^2)^{d/2}.
\]
This proves \prettyref{eq:point-small-ball}, including $\tau_z=0$.
\end{proof}

\subsection{Bounding the Lipschitz constant}
\label{sec:lip}

Recall that $\Lambda_R$ is the Euclidean Lipschitz constant of $z \mapsto D_zW_n$ on the sphere $\mathcal S_R$, where $D_zW_n$ is given by (\ref{eq:normalized-field}). We overbound it by the Lipschitz constant on the ball $B_R$, which is at most the supremum of the Frobenius norm of the Jacobian. Then we apply a Sobolev inequality on the ball to bound this supremum by integrals of squared derivatives (Sobolev norm), whose  expectation is easy to evaluate.

\begin{lemma}
\label{lem:sphere-variation}
For $R\ge1$, 
\begin{equation}
 \Expect\Lambda_R^2\le C_KR^{4d+8},
 \label{eq:direct-sphere-variation}
\end{equation}
where $C_K$ does not depend on $n$ or $G_\star\in\cP(K)$.
\end{lemma}

\begin{proof}
Let $Z(z)\equiv D_zW_n$, a smooth $\C^d$-valued function on $\R^{2d}$ (i.e., we identify the domain $\C^d$ as $\R^{2d}$). Denote its Jacobian by $J=\nabla_{\R^{2d}}Z \in \C^{d\times 2d}$. For a multi-index $\alpha\in\Nzero^{2d}$, write
$\partial^\alpha\equiv\partial^\alpha_{\R^{2d}}$; thus
$\partial^\alpha Z=(\partial^\alpha Z_\ell)_{\ell=1}^d$ is a vector
in $\C^d$.

To  bound $\sup_{B_R}\norm{J}_F^2$, we apply the Sobolev inequality to each real
and imaginary component of $y\mapsto J(Ry)$ on the interior of
$B_1\subset\R^{2d}$. 
Applying \cite[Theorem~8.8(iii)]{LiebLoss2001} (with $n=2d$, $p=2$,
$k=m=d+1$ and $\Omega=B_1$), we get
\[
 \sup_{B_R}\norm{J}_F^2
 \le C_d\sum_{|\beta|\le d+1}
       \int_{B_1}\norm{\partial_y^\beta[J(Ry)]}_F^2\d y.
\]
Since $\partial_y^\beta[J(Ry)]=R^{|\beta|}(\partial^\beta J)(Ry)$,
the change of variables $z=Ry$ gives
\begin{equation}\label{eq:Sobolev-ball}
 \Lambda_R^2\le\sup_{B_R}\norm{J}_F^2
 \le C_d\sum_{|\beta|\le d+1}R^{2|\beta|-2d}
       \int_{B_R}\norm{\partial^\beta J}_F^2\d z
 \le C_dR^2\sum_{1\le|\alpha|\le d+2}
       \int_{B_R}\norm{\partial^\alpha Z}^2\d z.
\end{equation}
Indeed, $J=\nabla Z$ so 
$\norm{\partial^\beta J}_F^2
=\sum_{j=1}^{2d}\norm{\partial^{\beta+e_j}Z}^2$, where $e_j$ is
the $j$th coordinate vector.

It remains to bound the second moments of the derivatives of $Z$
up to order $d+2$. 
Let $z=a+ib$ and
$h_z(x)\equiv w_zL_z(x)$, where 
$L_z$ is the likelihood ratio in \prettyref{eq:Lz}, 
$w_z :=\sqrt{r_\star(2\Rea z)}e^{-\norm{z}^2/2}$ as in \prettyref{eq:gradient-field}, 
and $r_\star = \frac{p_\star}{\varphi}$ as defined in \prettyref{eq:rG}.
Recalling \prettyref{eq:normalized-field}  and \prettyref{eq:Kz}, 
we have
\[
Z(z)
= D_zW_n = \sqrt{n} (\Expect_n - \Expect_\star) \mathbf{K}_z(X) = 
\sqrt{n} (\Expect_n - \Expect_\star) [h_z(X) (X-z)].
\]
Using the independence of the sample points, for every $\alpha$,
\[
 \Expect_\star\norm{\partial^\alpha Z(z)}^2
\le\Expect_\star\norm{\partial^\alpha\K_z(X)}^2=\Expect_\star\norm{\partial^\alpha[h_z(X)(X-z)]}^2. 
\]
The scalar factor satisfies, for every $j=1,\ldots,d$,
\begin{align*}
\partial_{a_j}h_z(x) &= \big(x_j-2a_j+(\partial_j\log r_\star)(2a)-ib_j\big)h_z(x),\\
\partial_{b_j}h_z(x) &= i(x_j-a_j)h_z(x) = i(x_j- 2a_j  + a_j)h_z(x).
\end{align*}
Using \prettyref{eq:logr-derivative}, each derivative of $\log r_\star$ is bounded, and hence for $R \geq 1$,
\begin{align*}
|\partial_{a_j} h_z(x)| + |\partial_{b_j} h_z(x)| \leq C(1 + |x_j - 2a_j| + R)|h_z(x)| \leq C_K R|h_z(x)|(1 + \norm{x-2a}).
\end{align*}
Differentiating repeatedly, we obtain,
for $R\ge1$, $z\in B_R$, and each fixed multi-index $\alpha$,
\[
 \norm{\partial^\alpha\K_z(x)}
 \le C_{K,\alpha}R^{|\alpha|+1}|h_z(x)|
                         (1+\norm{x-2a}^{|\alpha|+1}).
\]
By \prettyref{eq:change_measure},
$|h_z(x)|^2p_\star(x)=\rho_z(x-2a)$, where $\rho_z$ is in \prettyref{eq:second-moment-weight}.
Thus, with $q=|\alpha|+1$, the change of variables $t=x-2a$ gives
\[
 \Expect_\star\!\left[|h_z(X)|^2(1+\norm{X-2a}^q)^2\right]
 =\int_{\R^d}(1+\norm{t}^q)^2\rho_z(t)\,dt
 \le C_{K,\alpha},
\]
the last step applying $(1+\norm{t}^q)^2\le2(1+\norm{t}^{2q})$ and \prettyref{eq:rho-moment}.  In all, we obtain
\begin{equation}
 \Expect\norm{\partial^\alpha Z(z)}^2
 \le C_{K,\alpha}R^{2|\alpha|+2},\qquad z\in B_R.
 \label{eq:gradient-derivative-second-moment}
\end{equation}

Finally, we take expectations in \prettyref{eq:Sobolev-ball} and interchange
expectation and integration.  Since
$1\le|\alpha|\le d+2$ and $\text{vol}(B_R)=C_dR^{2d}$, applying
\prettyref{eq:gradient-derivative-second-moment} yields
$\Expect\Lambda_R^2
 \le C_KR^2R^{2d}R^{2d+6}
 =C_KR^{4d+8}$, as desired.
\end{proof}

\subsection{Proof of \prettyref{prop:centered-gradient-lower}}
\label{sec:pf-centered-gradient-lower}

We apply the discretized union bound \prettyref{eq:direct-sphere-union}:
\begin{align*}
\Prob\Big(\inf_{z\in\calS_R}\norm{D_zW_n}\le\eps\Big)
\le \sum_{z\in\N_\delta}\Prob\big(\norm{D_zW_n}\le2\eps\big)
+ \Prob\big(\Lambda_R>L\big).
\end{align*}
Let $R\geq 2$, $L\ge1$ and $\delta=\eps/L$. By \prettyref{eq:direct-sphere-variation} and Markov's inequality, the second term satisfies $\Prob\big(\Lambda_R>L\big) \leq C_KR^{4d+8}L^{-2}$.

For the first term, since $\Delta_{n,R}\le2\eps<1/2$, the small-ball estimate \prettyref{eq:point-small-ball} gives
\begin{align*}
 \sum_{z\in\N_\delta}\Prob\big(\norm{D_zW_n}\le2\eps\big)
 \le 
    CR^{2d} \eps^{2d} |\N_\delta| 
    + CR^{2d}\eps^d\sum_{z\in\N_\delta: \norm{\Ima z}\le1}
       1\wedge\left(\frac{\eps}{\norm{\Ima z}}\right)^d.
\end{align*}
The first contribution is at most
$CR^{4d-1}L^{2d-1}\eps$ 
by the cardinality bound \prettyref{eq:Ndelta} for $\N_\delta$.

To bound the second sum, we apply a peeling argument to handle the dependence of the small-ball probability on $\norm{\Ima z}$.
For $0\le t\le1$, let $M(t)=\#\{z\in\N_\delta:\norm{\Ima z}\le t\}$, which we bound using a volume calculation.
Since $\N_\delta$ is a $\delta$-packing, the balls
$B(z,\delta/3)$ centered at these points are disjoint and lie in the region
\[
 \left\{w\in\C^d:
 R-\delta/3\le \norm{w}\le R+\delta/3,\quad
 \norm{\Ima w}\le t+\delta/3\right\},
\]
which has volume at most $CR^{d-1}\delta(t+\delta)^d$. Indeed, with $A_\rho(u)$ denoting the surface area of $\{z \in \C^d: \norm{z} = \rho, \norm{\Ima z} \leq u\}$ in dimension $2d-1$, we have for $d\geq 2$
\begin{align*}
A_\rho(u) &= |S^{d-1}|\rho\int_{\norm{b}\leq u\wedge \rho} (\rho^2 - \norm{b}^2)^{(d-2)/2} \d b = |S^{d-1}|^2\rho \int_0^{u\wedge \rho} s^{d-1}(\rho^2-s^2)^{(d-2)/2}\d s\\
&\leq |S^{d-1}|^2 \rho^{d-1}  \int_0^{u\wedge \rho} s^{d-1}\d s \leq C_d \rho^{d-1}u^d,
\end{align*}
where $S^{d-1}$ denotes the unit sphere in $d$ dimensions, and $|S^{d-1}| = \frac{2\pi^{d/2}}{\Gamma(d/2)}$ is its surface measure in dimension $d-1$; for $d=1$, we have $A_\rho(u) = 4\rho\arcsin(\min\{u/\rho, 1\}) \leq 2\pi u$ for every $\rho > 0, u\geq 0$. Then the volume is bounded by $\int_{R-\delta/3}^{R+\delta/3} A_\rho(t+\delta/3) \d \rho \leq CR^{d-1}\delta(t+\delta)^d$, as claimed. Since the balls $\{B(z,\delta/3): z\in \calN_\delta\}$ are disjoint, dividing the volume bound by the $2d$-dimensional volume of $B(0,\delta/3)$ yields
\begin{equation}
 M(t)\le CR^{d-1}\delta^{-(2d-1)}(t+\delta)^d.
 \label{eq:direct-sphere-tube}
\end{equation}
Thus
\begin{align}
\sum_{z\in\N_\delta: \norm{\Ima z}\le1}
       1\wedge\left(\frac{\eps}{\norm{\Ima z}}\right)^d = \int_0^1 \left(1\wedge \frac{\eps}{t}\right)^d \d M(t) &= \eps^d M(1) + \int_{\eps}^1 \frac{d\eps^d}{t^{d+1}}M(t)\d t \notag \\
       &\overset{\prettyref{eq:direct-sphere-tube}}{\le} CR^{d-1}\delta^{-(2d-1)}\eps^d[1+\log(1/\eps)].  \label{eq:direct-sphere-peeling}
\end{align}
Combining these estimates in \prettyref{eq:direct-sphere-union} and substituting
$\delta=\eps/L$, we obtain
\begin{equation}\label{eq:direct-sphere-bound}
 \Prob\sth{\inf_{z\in\mathcal S_R}\norm{D_zW_n}\le\eps}
 \le C\{R^{4d+8}L^{-2}
       +L^{2d-1}\eps[R^{3d-1}\log(1/\eps)+R^{4d-1}]\}.
\end{equation}
Choosing $L=\eps^{-1/(2d+2)}$,
each term in \prettyref{eq:direct-sphere-bound} is at most
$CR^{4d+8}\eps^{1/(d+1)}$, proving \prettyref{eq:direct-sphere-final}.

\section{Upper bound for gradient growth}
\label{sec:growth}

We bound the supremum $M_{4R}$ of Proposition \ref{prop:jensen-critical-count} again using the decomposition $F_{n,G}=W_n-T_nf_{n,G}$ and machinery developed in \Cref{sec:shell,sec:centered-gradient}:  The Lipschitz estimate from \prettyref{sec:lip} controls the supremum of $\|D_zW_n\|$ on a sphere, and the population and empirical approximation bounds from
\prettyref{sec:shell} control $\|D_zT_nf_{n,G}\|$.

\begin{lemma}\label{lem:growth}
There are constants $C,c,R_0,n_0>0$, depending only on $K$, such that for every $G_\star\in\cP(K)$, $h>1$, $R\ge R_0$, and $n\ge n_0$ satisfying $\log n\ge C(R^2+\log h)$, with probability at least $1-n^{-c}-Ce^{-2R^2}$,
\begin{equation}
 \sup_{\substack{G\in\cP(K)\\\sqrt nH(p_G,p_\star)\le h}}
 \sup_{z\in\mathbb D_{4R}^d}\norm{\nabla F_{n,G}(z)}
 \le e^{CR^2}\left(h+\sqrt n\,e^{-cR^2\log R}\right).
 \label{eq:Fn-growth}
\end{equation}
For each integer $s\ge1$, there is a constant $C_{K,s}>0$, depending only on $K,s$, such that if $G_\star\in\cP_s(K)$ and $R\ge C_{K,s}$, then, with the same probability guarantee,
\begin{equation}
 \sup_{\substack{G\in\cP(K)\\\sqrt nH(p_G,p_\star)\le h}}
 \sup_{z\in\mathbb D_{4R}^d}\norm{\nabla F_{n,G}(z)}
 \le h e^{CR^2}.
 \label{eq:Fn-growth-atomic}
\end{equation}
\end{lemma}

\begin{proof}
Let $\tau=4\sqrt{d}R$ so that $\mathbb D_{4R}^d \subset B_\tau$. 
Taking the supremum over $G\in \cP(K)$ and $\sqrt nH(p_G,p_\star)\le h$, we have
\begin{align*}
\sup_{G,z\in\mathbb D_{4R}^d}\norm{\nabla F_{n,G}(z)} \overset{(i)}{\leq} Ce^{C_KR^2} \sup_{G,z\in \calS_\tau} \norm{D_z F_{n,G}(z)} \overset{(ii)}{\leq} Ce^{C_KR^2} \Big(\sup_{z\in \calS_\tau} \pnorm{D_z W_n}{} + \sup_{G,z\in \calS_\tau} \pnorm{D_z T_n f_{n,G}}{}\Big),
\end{align*}
where $(i)$ applies the maximum principle 
(see \cite[Chapter~II, \S 12, Theorem~2]{Shabat})
$\sup_{z\in B_\tau}\|\nabla F(z)\|
=
\sup_{z\in\mathcal S_\tau}\|\nabla F(z)\|$
and the lower bound in (\ref{eq:r-bound}); $(ii)$ follows from the decomposition \prettyref{eq:decomposition}.

To control the first term, let $Z(z)=D_zW_n$. Fix an arbitrary $z_0\in\mathcal S_\tau$. Recalling from \prettyref{sec:lip} that $\Lambda_\tau$ denotes the Lipschitz constant of $z \mapsto D_z W_n$ over $\calS_\tau$, we have
\[
 \sup_{z\in\mathcal S_\tau}\norm{Z(z)}
 \le\norm{Z(z_0)}+2\tau\Lambda_\tau.
\]
Using \prettyref{eq:gradient-derivative-second-moment} with $\alpha=0$
and \prettyref{lem:sphere-variation}, we obtain
\[
 \Expect\sup_{z\in\mathcal S_\tau}\norm{Z(z)}^2
 \le2\Expect\norm{Z(z_0)}^2+8\tau^2\Expect\Lambda_\tau^2
 \le C_K\tau^{4d+10}.
\]
Consequently, Markov's inequality gives
\[
 \Prob\left\{\sup_{z\in\mathcal S_\tau}\norm{D_zW_n}
                   >e^{2R^2}\right\}
 \le C_K\tau^{4d+10}e^{-4R^2}
 \le C_Ke^{-2R^2},
\]
for large enough $R$.

For the second term, we apply \prettyref{lem:population-shell} and \prettyref{prop:empirical-shell} at radius $\tau$.  The condition on the sample size holds for sufficiently large $C$, since $\tau^2=16dR^2$. By the decomposition $T_nf_{n,G}=Tf_{n,G}+V_{n,G}$ in (\ref{eq:empirical-decomp}), except on an event of probability at
most $n^{-c}$,
\[
 \sup_{\substack{G\in\cP(K)\\\sqrt nH(p_G,p_\star)\le h}}
 \sup_{z\in\mathcal S_\tau}\norm{D_zT_nf_{n,G}}
 \le Ch e^{-c\tau} + C\sqrt n\,e^{-c\tau^2\log\tau} + n^{-c}.
\]
If $G_\star\in\cP_s(K)$ and $R\ge C_{K,s}$, then $\tau\ge R\ge C_{K,s}$, so the middle term $C\sqrt n\,e^{-c\tau^2\log\tau}$ is not necessary by \prettyref{eq:population-shell-atomic}.

Combining the two estimates yields
\begin{align*}
 \sup_{\substack{G\in\cP(K)\\\sqrt nH(p_G,p_\star)\le h}}
   \sup_{z\in\mathbb D_{4R}^d}\norm{\nabla F_{n,G}(z)}
 & \le e^{C_KR^2}
   \left(e^{2R^2}+Ch e^{-c\tau}
                 +C\sqrt n\,e^{-c\tau^2\log\tau}+n^{-c}\right)\\
 &\le e^{CR^2}\left(h+\sqrt n\,e^{-cR^2\log R}\right).
\end{align*}
Here $h>1$ and $\tau=4\sqrt d\,R$ allow the fixed factors and
$e^{2R^2}$ to be absorbed into $e^{CR^2}$. This proves \prettyref{eq:Fn-growth}.  Omitting the term $C\sqrt n\,e^{-c\tau^2\log\tau}$ proves \prettyref{eq:Fn-growth-atomic}, with the same failure probability.
\end{proof}

\section{Proof of the main theorems}
\label{sec:proofs}

We first record a consequence of \prettyref{prop:centered-gradient-lower} used in both proofs. Fix small enough $b>0$ (see (\ref{eq:general-fitted-sphere}) below for details) and $C_R<\infty$. Since $L_n=o(\log n)$, for all sufficiently large $n$ and all $R\ge R_0$ with $R^2\le C_RL_n$, we have
\[
\Delta_{n,R}=Cn^{-1/6}e^{CR^2}
\le e^{-2bR}<1/4,
\]
since $\log(\Delta_{n,R}/e^{-2bR})
= -\frac16\log n + CR^2+2bR+O(1)
\to-\infty$ uniformly in this range. Applying \prettyref{eq:direct-sphere-final} with
$\eps=e^{-2bR}$ and absorbing the polynomial factor $R^{4d+8}$ therein into the exponential (in $R$) term $\eps^{1/(d+1)}$ gives
\begin{equation}\label{eq:centered-gradient-lower}
\Prob\Big(\inf_{z\in\mathcal S_R}
\norm{D_zW_n}\le e^{-2bR} \Big) \le Ce^{-cR},
\end{equation}
where $C,c$ may depend on $K,b,C_R$ but not on $n,R$, or $G_\star\in\cP(K)$.

\subsection{Proof of Theorem~\ref{thm:continuous-prior}}
\label{sec:arbitrary}

To prove \prettyref{thm:continuous-prior}, we need to bound the probability $a_n(h)$ of the bad Hellinger event (\ref{eq:fitted-hellinger-failure}) required by Proposition \ref{prop:fitted-shell}. This ingredient is supplied by the following lemma.

\begin{lemma}\label{lem:universal-H}
There are $C,c>0$, depending only on $K$, such that for all sufficiently large $n$,
\begin{equation}
 \sup_{G_\star\in\cP(K)}
 \Prob_{G_\star}\left\{
   \sqrt n\,H(p_{\widehat G},p_\star)>C(\log n)^{(d+1)/2}
 \right\}\le n^{-c}.
 \label{eq:universal-H}
\end{equation}
\end{lemma}

\begin{proof}
Applying \cite[Theorem~2.1]{SahaGuntuboyina2020} with $S=K$, $\alpha=\beta=1/2$, $t=2$, and truncation parameter $M_n=\sqrt{10\log n}$ (the distance-moment term $\mu_p$ in their equation (2.3) vanishes since $G_\star(K)=1$), we have
$\sqrt nH(p_{\widehat G},p_\star)\le C_K(\log n)^{(d+1)/2}$ with failure
probability at most $2n^{-4}$, uniformly over $G_\star\in\cP(K)$.
This implies \prettyref{eq:universal-H} for sufficiently large $n$.
\end{proof}
\begin{remark}[Bootstrapping the Hellinger guarantee] 
We only require much weaker density estimation results than \prettyref{lem:universal-H}. Indeed, for our proof of \prettyref{thm:continuous-prior} below to go through, the factor $(\log n)^{(d+1)/2}$ in \prettyref{lem:universal-H} can be replaced by $e^{o(\sqrt{L_n})}$. In other words, even a weak statistical guarantee on $\sqrt nH(p_{\widehat G},p_\star)$ suffices to establish the structural property $|\supp(\widehat G)|=O_{\Prob}(L_n^d)$, which in turn yields a sharper upper bound on $\sqrt nH(p_{\widehat G},p_\star)$.
\end{remark}

\begin{proof}[Proof of \prettyref{thm:continuous-prior}]
We apply the general bound in \prettyref{prop:fitted-shell}.  Fix $B\ge B_0$ and set
\begin{equation}
 h=C_K(\log n)^{(d+1)/2},\qquad R=B\sqrt{L_n},
 \label{eq:general-scales}
\end{equation}
where $C_K$ is the constant in \prettyref{lem:universal-H}, and that lemma gives $a_n(h)\le n^{-c}$.  Since $R^2+\log h=O(L_n)=o(\log n)$, the sample-size condition \prettyref{eq:fitted-sample-condition} holds for sufficiently large $n$, hence \prettyref{prop:fitted-shell} gives
\begin{equation*}
\sup_{z\in\mathcal S_R} \norm{D_zT_nf_{n,\widehat G}} \le Ch e^{-cR}+C\sqrt n\,e^{-cR^2\log R}+n^{-c}.
\end{equation*}
Choose $b > 0$ to be a sufficiently small constant depending only on $K$, and we claim that the three error terms on the right are bounded by $e^{-8bR}$. Indeed, $R\gg\log\log n$ implies $Ch e^{-cR}\le e^{-8bR}$, while $R^2\log R\asymp B^2\log n$ implies
\[
 C\sqrt n\,e^{-cR^2\log R}\le n^{-c'}
\]
for some $c'>0$ after choosing $B_0$ sufficiently large, and lastly, both $n^{-c'}$ and $n^{-c}$ are at most $e^{-8bR}$ for sufficiently large $n$.  Thus 
\begin{equation}\label{eq:general-fitted-sphere}
\sup_{\norm{z}=R}\norm{D_zT_nf_{n,\widehat G}} \le 2e^{-4bR}
\end{equation}
with probability at least $1-n^{-c}$ for small constant $c$.
On the other hand, (\ref{eq:centered-gradient-lower}) shows that with probability $1-Ce^{-cR}$, we have $\inf_{\norm{z}{} = R} \pnorm{D_z W_n}{} \geq e^{-2bR}$. On the intersection of these two events, \prettyref{eq:gradient-triangle} gives
\[
 \inf_{\norm{z}=R}\norm{D_zF}
 \ge e^{-2bR}-2e^{-4bR}
 \ge\frac12e^{-2bR}
\]
for sufficiently large $R$. By (\ref{eq:normalization_connection}), we then have
\begin{align}\label{eq:m_R_lower}
 m_R = \inf_{z\in \calS_R} \norm{\nabla F_{n,\widehat G}(z)} \ge \frac{1}{2}e^{R^2/2-(\gamma_K+2b)R}.
\end{align}

Next, we apply \prettyref{lem:growth} with the same $h,R$. For sufficiently large $B$, the remainder $\sqrt n\,e^{-cR^2\log R}$ in \prettyref{eq:Fn-growth} is at most $1$
by the same calculation above. Since
$\log h=O(\log\log n)=o(R^2)$, we obtain
\begin{equation} \label{eq:general-growth}
M_{4R} = \sup_{z\in\mathbb D_{4R}^d}\norm{\nabla F_{n,\widehat G}(z)} \le e^{C_KR^2}
\end{equation}
with failure probability at most $n^{-c}+C_Ke^{-2R^2}$.

Together, the estimates of $m_R$ and $M_{4R}$ give
$0\le\log(M_{4R}/m_R)\le C_KR^2$, hence \prettyref{prop:jensen-critical-count} with $\rho=4R$ bounds the number of interior critical points by $C_KR^{2d}$.  By \prettyref{prop:boundary-count} (with its condition (\ref{eq:boundary-conditions}) satisfied by \prettyref{eq:general-fitted-sphere} and \prettyref{eq:general-growth}), the number of boundary atoms is at most $C_KR^{2(d-1)}$, with additional failure probability bounded by $Ce^{-cR}$.
Therefore,
\[
 |\supp(\widehat G)|\le C_KR^{2d}\le C_KB^{2d}L_n^d,
\]
with failure probability at most $Ce^{-cR}+n^{-c}=Ce^{-cB\sqrt{L_n}}+n^{-c}$, which proves \prettyref{eq:tail-continuous-prior} and hence \prettyref{eq:OP-continuous-prior}.
\end{proof}

\subsection{Proof of Theorem~\ref{thm:main}}
\label{sec:proof-main}

Fix a finitely supported $G_\star\in\cP(K)$ and let $s=|\supp(G_\star)|$. By \cite[Theorem~2.1(1)]{ZhangVolgushev2026},
\[
n \chi^2(p_{\widehat G} \| p_\star) = n\int_{\R^d}\frac{(p_{\widehat G}-p_\star)^2}{p_\star}=O_{\Prob}(1).
\]
 Since $H^2(p,q)\le \chi^2(p\|q)$, it follows that
\begin{equation}
 \sqrt n\,H(p_{\widehat G},p_\star)
 =O_{\Prob}(1).
 \label{eq:finite-prior-H}
\end{equation}
Given $\eps>0$, choose a deterministic $h>1$ such that the failure probability in \prettyref{eq:fitted-hellinger-failure} satisfies
\[
 \limsup_{n\to\infty}a_n(h)<\eps/3.
\]
Next choose a fixed $R$ sufficiently large so that 
(a) $K\subset B_R$;
(b) all lower bounds on $R$ in \prettyref{prop:fitted-shell}, \prettyref{prop:centered-gradient-lower}, \prettyref{lem:growth}, \prettyref{prop:boundary-count} hold; (c)
\begin{equation}\label{eq:finite-prior-radius}
 R\ge C_{K,s},\qquad
 \log h\le R^2,\qquad
 Ch e^{-cR}\le e^{-8bR};
\end{equation}
and (d) the combined failure probability from \prettyref{prop:centered-gradient-lower}, \prettyref{lem:growth}, and \prettyref{prop:boundary-count}, namely, $n^{-c} + Ce^{-cR}$, is bounded by $\eps/3$ for large enough $n$.
With $h$ and $R$ independent of $n$, the sample-size conditions in \prettyref{prop:fitted-shell} and \prettyref{lem:growth}, and the radius condition $R^2\le C_R L_n$ at the beginning of \prettyref{sec:proofs} hold for all sufficiently large $n$. 
Moreover, $n^{-c}\le e^{-8bR}$ for large $n$.  Thus, 
in view of \prettyref{eq:finite-prior-radius},
applying  \prettyref{eq:fitted-shell-atomic} 
in \prettyref{prop:fitted-shell} for 
finitely supported $G_\star$, we obtain
\begin{equation}\label{eq:atomic-fitted-sphere}
\sup_{z\in\mathcal S_R}\norm{D_zT_nf_{n,\widehat G}}
\le Ch e^{-cR}+n^{-c}\le 2e^{-4bR}
\end{equation}
with probability at least $1-a_n(h)-n^{-c}$.
On the other hand, \prettyref{eq:centered-gradient-lower} shows that with probability at least $1-Ce^{-cR}$,
\[
\inf_{z\in\mathcal S_R}\norm{D_zW_n}\ge e^{-2bR}.
\]
On the intersection of these events, \prettyref{eq:gradient-triangle} gives
\[
\inf_{z\in\mathcal S_R}\norm{D_zF_{n,\widehat G}}
\ge e^{-2bR}-2e^{-4bR}
\ge\frac12e^{-2bR}
\]
for sufficiently large $R$.  By \prettyref{eq:normalization_connection},
exactly as in the proof of \prettyref{thm:continuous-prior}, we obtain
\begin{equation}\label{eq:atomic-mR-lower}
m_R=\inf_{z\in\mathcal S_R}\norm{\nabla F_{n,\widehat G}(z)}
\ge\frac12e^{R^2/2-(\gamma_K+2b)R}.
\end{equation}

Next, we apply \prettyref{lem:growth} with the same $h,R$.  By the
finite-atomic bound \prettyref{eq:Fn-growth-atomic} and $\log h\le R^2$,
\begin{equation}\label{eq:atomic-growth}
M_{4R}=\sup_{z\in\mathbb D_{4R}^d}
\norm{\nabla F_{n,\widehat G}(z)}
\le h e^{CR^2}\le e^{C_KR^2}
\end{equation}
on the Hellinger event $\{\sqrt n\,H(p_{\widehat G},p_\star)\leq h\}$, except for an additional event 
of probability at most $n^{-c}+Ce^{-2R^2}$.

Combining \prettyref{eq:atomic-mR-lower} and \prettyref{eq:atomic-growth}, we obtain 
$0\le\log(M_{4R}/m_R)\le C_KR^2$.  Hence
\prettyref{prop:jensen-critical-count} with $\rho=4R$ bounds the number of interior critical points by $C_KR^{2d}$. By \prettyref{prop:boundary-count} (with its condition (\ref{eq:boundary-conditions}) satisfied by \prettyref{eq:atomic-fitted-sphere} and
\prettyref{eq:atomic-growth}), there are at most $C_KR^{2(d-1)}$ boundary atoms with failure probability bounded by $Ce^{-cR}$. Consequently,
\[
|\supp(\widehat G)|\le C_KR^{2d},
\]
with probability at least $1-a_n(h)-Ce^{-cR}-n^{-c}$. By our choices of $h$ and $R$ (both depending on $\eps$ but not on $n$), we have 
\[
\limsup_{n\to\infty}
\Prob\big(|\supp(\widehat G)|> C_KR^{2d}\big) \leq \eps.
\]
 Since $\eps>0$ was arbitrary, this proves
\prettyref{eq:OP-finite-prior}, namely, 
$|\supp(\widehat G)|=O_\Prob(1)$.

\appendix

\section{Boundary case}
\label{app:boundary}

We count the critical points of the KKT function (\ref{eq:KKT-function}) on the boundary $\partial K$, or equivalently, on the relative interior of each face. Since $K=[-M,M]^d$ is invariant under signed coordinate permutations, it suffices for the probabilistic argument to consider the representative $k$-dimensional face
\begin{align}\label{eq:representative_face}
 H_k=[-M,M]^k\times\{c_k\},\qquad
 c_k=M\mathbf 1_{d-k},\qquad 1\le k<d.
\end{align}

\begin{proposition}
\label{prop:boundary-count}
Fix $b>0$, $C_g>0$, and $C_R<\infty$.  There are $R_0,C,c_0>0$,
depending only on $K,b,C_g$, and $n_0=n_0(K,b,C_g,C_R)$ with the
following property.  For every $n\ge n_0$, every $R\ge R_0$ with
$R^2\le C_RL_n$, and every $G_\star\in\cP(K)$, there is an event
$\mathcal E_{\partial,n,R}$ with probability at least $1-Ce^{-c_0R}$, such that on this event,
if the NPMLE $\widehat G$ satisfies
\begin{equation} \label{eq:boundary-conditions}
\sup_{\norm{z}=R}\norm{D_zT_nf_{n,\widehat G}}\le2e^{-4bR}, \qquad \sup_{z \in \mathbb D_{4R}^d}\norm{\nabla F_{n,\widehat G}}\le e^{C_gR^2},
\end{equation}
then
\begin{equation}\label{eq:boundary-count}
|\supp(\widehat G)\cap\partial K|\le CR^{2(d-1)}.
\end{equation}
\end{proposition}

We need the following variant of Proposition \ref{prop:centered-gradient-lower} on the face \prettyref{eq:representative_face}. In the application below, we pick $\eps = e^{-2bR}$ as in (\ref{eq:centered-gradient-lower}).

\begin{lemma}\label{lem:edge-derivative}
There are constants $R_0,C>0$, depending only on $K$, such that the following holds. For $1\le k\le d$, let
\[
\rho_k=\sqrt{R^2-(d-k)M^2}.
\]
For every $R\ge R_0$ and $0<\eps<1/4$ satisfying
$\eps\ge\Delta_{n,R}$,
we have, uniformly over $G_\star\in\cP(K)$,
\begin{equation}\label{eq:edge-derivative}
\Prob\left(\inf_{\norm{u}=\rho_k}w_{(u,c_k)}\norm{\nabla_uW_n(u,c_k)} \le\eps \right) \le CR^{4d+8}\eps^{1/(k+1)},
\end{equation}
where $w_{(u,c_k)}$ is the weight $w_z$ defined in \prettyref{eq:wz} applied to $z = (u,c_k)$.
\end{lemma}

\begin{proof}
Abbreviate $c\equiv c_k$, $\rho\equiv\rho_k$, and $Z(z)\equiv D_zW_n$. 
Applying $Z(z)$ to a point $z=(u,c)$ on the face $H_k$, define
\[ Z_k(u) := P_kZ(u,c)=w_{(u,c)}\nabla_uW_n(u,c),
\]
where $P_k:\C^d\to\C^k$ projects onto the first $k$ free coordinates.
For $\norm{u}=\rho$ we have $\norm{(u,c)}=R$, and $\rho\asymp R$. We follow the same pointwise small-ball estimate plus covering argument of Proposition \ref{prop:centered-gradient-lower}.

\emph{Pointwise small-ball estimate.}
Let $z=(u,c)$, $\phi=\Rea z\cdot\Ima z$, and
$\tau_u=1\wedge\norm{\Ima u}$. 
Taking the principal submatrix corresponding to the first $k$ real and imaginary coordinates of (\ref{eq:gradient-covariance-bounds}), we obtain
\[
 \Cov_\star\!\begin{pmatrix}
   \Rea(e^{-i\phi}P_k\K_z)\\
   \Ima(e^{-i\phi}P_k\K_z)
 \end{pmatrix}
 \succeq c(1+R)^{-2}\diag(I_k,\tau_u^2I_k),
\]
which also holds if $u$ is real. Moreover, $\Expect_\star\norm{P_k\K_z}^3\le e^{CR^2}$
by \prettyref{eq:third-moment}.
Applying \prettyref{lmm:esseen} in dimensions $2k$ therefore gives
\begin{equation}\label{eq:edge-small-ball}
\Prob_{G_\star}\big(\norm{Z_k(u)}\le\eps\big)
\le CR^{2k}\frac{\eps^{2k}}{(\eps^2+\tau_u^2)^{k/2}}, \qquad \Delta_{n,R}\le\eps<\frac{1}{2}.
\end{equation}

\emph{Lipschitz estimate.}
The embedding $u\mapsto(u,c)$ is an isometry from
$\{\norm{u}=\rho\}$ into $\mathcal S_R$.  Since $P_k$ is a
contraction, the Lipschitz constant $\Lambda_k$ of $Z_k$ on this sphere satisfies $\Lambda_k\le\Lambda_R$. Consequently,
\begin{equation}
 \Expect\Lambda_k^2\le CR^{4d+8}
 \label{eq:edge-lipschitz}
\end{equation}
by \prettyref{lem:sphere-variation}.

\emph{Uniform lower bound.}
We apply the same discretization and peeling argument in
\prettyref{sec:pf-centered-gradient-lower}, using \prettyref{eq:edge-small-ball}
in complex dimension $k$ and \prettyref{eq:edge-lipschitz} for the
Lipschitz term.  The sphere has real dimension $2k-1$ and
radius $\rho\asymp R$.  With mesh $\delta=\eps/L$, the analogue of (\ref{eq:direct-sphere-bound}) yields, for $L\ge1$ and
$\Delta_{n,R}\le\eps<1/8$,
\begin{align*}
 \Prob_{G_\star}\left\{
   \inf_{\norm{u}=\rho}\norm{Z_k(u)}\le\eps\right\}
 \le C\Big\{R^{4d+8}L^{-2}
  +L^{2k-1}\eps
    \big[R^{3k-1}\log(e/\eps)+R^{4k-1}\big]\Big\}.
\end{align*}
Taking $L=\eps^{-1/(2k+2)}$ bounds the last display by $CR^{4d+8}\eps^{1/(k+1)}$, as desired.
\end{proof}

\begin{proof}[Proof of \prettyref{prop:boundary-count}]
If $M=0$ or $d=1$, the boundary contains at most two points, so there is nothing to prove.  Suppose $M>0$ and $d\ge2$. For every proper face $H$ of dimension $1\le k<d$, after permuting the coordinates and reflecting those that equal to $-M$, it can be mapped to the representative face $H_k$
in (\ref{eq:representative_face}).
Since permutation and reflection are orthogonal transforms, \prettyref{lem:edge-derivative} with the choice $\eps = e^{-2bR}$ applies to the event
\[
 \inf_{\norm{u}=\rho_k}w_{(u,c_H)}
             \norm{\nabla_uW_n(u,c_H)}>e^{-2bR}
\]
where 
$\rho_k=\sqrt{R^2-(d-k)M^2}$.
There are $2^{d-k}\binom dk$ faces of dimension $k$ of the cube. Taking the intersection and applying the union bound, we obtain an event $\mathcal E_{\partial,n,R}$ with $\Prob_{G_\star}(\mathcal E_{\partial,n,R}^c)\le Ce^{-c_0R}$.

On $\mathcal E_{\partial,n,R}$, suppose that $\widehat G$ satisfies \prettyref{eq:boundary-conditions}, and let $F_H(u) := F_{n,\widehat G}(u,c_H)$. Since
$\norm{(u,c_H)}=R$ whenever $\norm{u} = \rho_k$,
the decomposition \prettyref{eq:decomposition} yields
\[
 \inf_{\norm{u}=\rho_k}w_{(u,c_H)}\norm{\nabla F_H(u)}
 \ge e^{-2bR}-2e^{-4bR}\ge\frac12e^{-2bR}.
\]
On this sphere $w_{(u,c_H)}\le e^{\gamma_KR-R^2/2}$ by (\ref{eq:r-bound}), and $\mathbb D_{4\rho_k}^k\times\{c_H\}\subset\mathbb D_{4R}^d$. The quantities $M_{4\rho_k}$ and $m_{\rho_k}$
of \prettyref{prop:jensen-critical-count} in dimension $k$
therefore satisfy
\begin{align}
 \inf_{\norm{u}=\rho_k}\norm{\nabla F_H(u)}&\ge e^{-CR^2}, 
 \label{eq:edge-lower} \\
 \sup_{u\in\mathbb D_{4\rho_k}^k}\norm{\nabla F_H(u)}
 &\le e^{C_gR^2}.
 \label{eq:edge-upper}
\end{align}
Consequently,
\[
 \#\{u\in\C^k:\norm{u}<\rho_k,\ \nabla F_H(u)=0\}
 \le C\left[1+\frac{(C_g+C)R^2}{\log4}\right]^k
 \le CR^{2k}.
\]
This bounds the number of critical points of $F_H$ over $[-M,M]^k$, since  $kM^2+(d-k)M^2=dM^2=\gamma_K^2<R^2$.  By
\prettyref{lem:KKT}, each atom in the relative interior of $H$ is a critical point of $F_H$.  Hence
$|\supp(\widehat G)\cap\operatorname{relint}H|\le CR^{2k}$.

Lastly, note that the relative interiors of the proper faces partition $\partial K$, and vertices contribute at most one point each. Therefore
\[
 |\supp(\widehat G)\cap\partial K|
 \le 2^d+\sum_{k=1}^{d-1}2^{d-k}\binom dk\,CR^{2k}
 \le CR^{2(d-1)}.
\]
This proves \prettyref{eq:boundary-count}.
\end{proof}

\section{Proof of auxiliary results}
\label{app:aux}

\subsection{Proof of \prettyref{prop:jensen-critical-count}}
\label{app:critical-count}

We recall the setup of \prettyref{prop:jensen-critical-count} so that this subsection is self-contained.  For $R>0$ and $\rho\ge2R$, let $B_R=\{z\in\C^d:\norm{z}\leq R\}$ and $\mathbb D_\rho^d=\{z\in\C^d:|z_j|\le\rho,\ 1\le j\le d\}$.  Suppose that $F$ is holomorphic on a neighborhood of $\mathbb D_\rho^d$ and set
\[
 M_\rho=\sup_{z \in \mathbb D_\rho^d}\norm{\nabla F(z)},
 \qquad m_R=\inf_{\norm{z}=R}\norm{\nabla F(z)}>0.
\]
We will prove the claim in (\ref{eq:jensen-critical-count}): counting multiplicity,
\begin{equation*}
\#\{z\in B_R:\nabla F(z)=0\}
\le C_d\left[1+\frac{\log(M_\rho/m_R)}{\log(\rho/R)}\right]^d,
\end{equation*}
for some $C_d > 0$ that only depends on $d$. The inclusion $\mathcal S_R\subset\mathbb D_\rho^d$ ensures $M_\rho\ge m_R$, so the logarithm is nonnegative.

We make two comments before presenting its proof:
\begin{enumerate}
\item The bound \prettyref{eq:jensen-critical-count} is essentially a special case of \cite[Theorem 1.7]{BPPSS2024}. More precisely, thanks to condition $m_R > 0$, for small enough $\delta$ (e.g. $\delta = m_R/4$), all connected components of the sub-level set $E_\delta(\nabla F, R) = \{z\in B_R: \pnorm{\nabla F}{}\leq \delta\}$ are \emph{islands} in the language of \cite[Section 1.4]{BPPSS2024}, and hence with $\tau(f,r,\delta)$ defined after Theorem 1.6 therein, their Theorem 1.7 implies that
\begin{align*}
\#\{z\in B_R:\nabla F(z)=0\} \leq \tau(\nabla F, R, m_R/4) \leq C(d,\rho/R) \Big[\log\frac{4M_\rho}{m_R}\Big]^d.
\end{align*}
Our claimed bound (\ref{eq:jensen-critical-count}) makes the dependence on $\rho/R$ explicit.
\item It is instructive to explain the difference between (\ref{eq:jensen-critical-count}) and its one-dimensional analogue. For example, \cite[Lemma 4]{PolyanskiyWu2020}, when applied with $R,2R,4R$, bounds the number of critical points of one-dimensional $F$ in a disk of radius $R$ as follows:
\begin{align}
\#\{z\in B_R:F'(z)=0\} \leq \frac{\log(M_{4R}/M_{2R})}{\log(3/2)}.
\label{eq:1DJensen}
\end{align}
In contrast to (\ref{eq:jensen-critical-count}), the above one-dimensional formula only requires the maximum modulus quantities $M_{2R}$ and $M_{4R}$. 
However, it is known that in multiple dimensions, controlling the maximum modulus alone is not enough for bounding the number of zeros or critical points \cite{BPPSS2024}.
Indeed, consider the following example for $d=2$ and $R=1$. For any positive integer $N$, let
\begin{align*}
q_N(w) = \prod_{j=0}^{N-1} \Big(w - \frac{j}{N}\Big),
\end{align*}
and
\begin{align*}
F_N(z_1,z_2) = \frac{z_1^2}{2} + 5^{-N}\int_0^{z_2} q_N(w)\d w,
\end{align*}
Then
\begin{align*}
\nabla F_N(z_1,z_2) = \Big(z_1, 5^{-N}q_N(z_2)\Big),
\end{align*}
and $F_N$ has exactly $N$ simple critical points $(0, j/N)$ for $j=0,\ldots,N-1$, all of which belong to $B_1$. On the other hand, using $|q_N(w)| \leq 5^N$ for $|w|\leq 4$, we have
\begin{align*}
M_{4} = \sup_{z \in \mathbb D_4^2}\pnorm{\nabla F_N(z)}{} \leq \sqrt{17}, \quad M_2 =  \sup_{z \in \mathbb D_2^2}\pnorm{\nabla F_N(z)}{} \geq \pnorm{\nabla F_N((2,0))}{} = 2.
\end{align*}
Consequently, $M_4/M_2 \leq 3$ for all $N$, but there are $N$ critical points inside $B_1$. This counterexample is essentially due to the fact that the first coordinate of $\nabla F_N(z)$ is dominant, while the second controls the number of critical points, which is not captured by maximal modulus $M_R$ that only controls the norm of the whole gradient vector; in contrast, $\log(1/m_1) \asymp N$, so (\ref{eq:jensen-critical-count}) gives a non-trivial upper bound $O(N^2)$. 
\end{enumerate}

The proof of \prettyref{prop:jensen-critical-count} follows the same program as \cite[Theorem 1.7]{BPPSS2024}: we approximate the components of the gradient by polynomials, use the boundary minimum to apply Rouch\'e's theorem, and then apply B\'ezout's theorem to the polynomial approximant. The last two steps are recorded in the following lemma.

\begin{lemma}
\label{lem:rouche-bezout}
Let $\Omega\subset\C^d$ be a bounded open set with smooth compact boundary, allowing
finitely many components.  Let $f,g:\C^d\to\C^d$ be holomorphic near
$\overline\Omega$, and suppose
\begin{equation}
 \norm{f-g}<\norm{g}\qquad\text{on }\partial\Omega.
 \label{eq:rouche-gap-geometry}
\end{equation}
Then $f$ and $g$ have the same finite number of zeros (counting multiplicity) in $\Omega$.
Furthermore, if $g=(g_1,\ldots,g_d)$ is a polynomial and each component has degree at most $N$, then this number is at most $N^d$.
\end{lemma}

\begin{proof}
The statement is the standard multivariate Rouch\'e's theorem (cf.~e.g.~\cite[Chapter~IV, \S18, Theorem~2]{Shabat}).
The finiteness of the zero set is shown in the proof therein and follows from the non-vanishing of both $f$ and $g$ on the boundary
$\partial\Omega$ and applying \cite[Theorem~7, \S8.24]{Shabat}.

If $g=(g_1,\ldots,g_d)$ is a polynomial, its zeros in $\Omega$ are isolated by the
finiteness established above.  By B\'ezout's theorem, their total multiplicity is at most $N^d$.
\end{proof}

\begin{proof}[Proof of \prettyref{prop:jensen-critical-count}]
Let $\lambda=\rho/R\ge2$, $M=M_\rho$, $m=m_R$, and define the holomorphic function
\[
 g(w)=\nabla F(Rw),\qquad w\in\mathbb D_\lambda^d.
\]
Then $\sup_{\mathbb D_\lambda^d}\norm{g}=M$ and $\inf_{\norm{w}=1}\norm{g(w)}=m$.  Expand $g(w)=\sum_\alpha c_\alpha w^\alpha$, with vector coefficients $c_\alpha\in\C^d$ and multi-indices $\alpha\in\Nzero^d$.
Applying Cauchy's integral formula iteratively (cf.~e.g.~\cite[Chap.~1, \S5, Theorem 1]{Shabat}) gives
\[
 c_\alpha=\frac{1}{(2\pi i)^d}
 \oint_{|\zeta_1|=\lambda}\cdots\oint_{|\zeta_d|=\lambda}
 \frac{g(\zeta_1,\ldots,\zeta_d)}
      {\prod_{j=1}^d\zeta_j^{\alpha_j+1}}
 \d \zeta_d\cdots \d \zeta_1.
\]
Taking the Euclidean norm and using $\norm{g(\zeta)}\le M$ on these contours, we get
\[
 \norm{c_\alpha}\le M\lambda^{-|\alpha|}.
\]
Let $p_N(w)=\sum_{|\alpha|\le N}c_\alpha w^\alpha$.  There are $\binom{\ell+d-1}{d-1}$ multi-indices of total degree $\ell$, and $|w_j|\le1$ on $B_1$, so
\begin{align}\label{eq:gradient-Taylor-remainder}
 \notag\sup_{B_1}\norm{g-p_N}
 &\le M\sum_{\ell>N}\binom{\ell+d-1}{d-1}\lambda^{-\ell} = M\lambda^{-(N+1)}\sum_{j=0}^\infty \binom{N+j+d}{d-1}\lambda^{-j}\\
 &\overset{(i)}{\le} \frac{M(N+d)^{d-1}\lambda^{-(N+1)}}{(d-1)!} \sum_{j=0}^\infty (1+j)^{d-1} \lambda^{-j} \overset{(ii)}\leq  C_dM(N+d)^{d-1}\lambda^{-(N+1)},
\end{align}
where $(i)$ applies the estimate $\binom{N+j+d}{d-1} \leq \frac{(N+j+d)^{d-1}}{(d-1)!}$, and $(ii)$ applies the condition that $\lambda \geq 2$.

Choose
\begin{equation} \label{eq:gradient-Taylor-degree}
 N=\left\lceil N_0+
     \frac{8\log(M/m)}{\log\lambda}\right\rceil,
\end{equation}
where $N_0\ge2$ is sufficiently large.  Uniformly in $N\ge N_0$ and $\lambda\ge2$, the factor $C(N+d)^{d-1}$ is then at most $\lambda^{N/2}$.  Thus the right side of \prettyref{eq:gradient-Taylor-remainder} is at most $M\lambda^{-N/2-1}\le m/3$, as long as $N_0$ is sufficiently large.  In particular, on $\partial B_1$,
\[
 \norm{p_N}\ge\frac23m,
 \qquad
 \norm{g-p_N}\le\frac13m<\norm{p_N}.
\]
Apply \prettyref{lem:rouche-bezout} to $g$ and $p_N$ on the open ball $\Omega=\operatorname{int}B_1$.  Each polynomial component has degree at most $N$, so $g$ has at most $N^d$ zeros in the interior of $B_1$, counting multiplicity.  Since $\norm{g}\ge m>0$ on $\partial B_1$, the same bound holds in $B_1$, which is precisely the critical points of $F$ in $B_R$ after the change of variables $z=Rw$. 
\end{proof}

\subsection{Proof of Lemma~\ref{lem:polarization}}
\label{app:mixed-moments}

We first prove \prettyref{eq:low-coeff} by Cauchy--Schwarz, then use the resulting control of the first $2s$ moments in the moment-comparison inequality in \cite[Lemma~6.3]{ZhangVolgushev2026} to prove \prettyref{eq:mixed-high}.

Write $H=H(p_G,p_\star)$.  For densities $p_1,p_2$ and a real function $f$ square-integrable under both, Cauchy--Schwarz gives
\begin{align}
 \left|\int f(p_1-p_2)\right|
 &=\left|\int f(\sqrt{p_1}-\sqrt{p_2})(\sqrt{p_1}+\sqrt{p_2})\right|\notag\\
 &\le H(p_1,p_2)\left\{\int f^2(\sqrt{p_1}+\sqrt{p_2})^2\right\}^{1/2}\notag\\
 &\le H(p_1,p_2)\left\{2\int f^2(p_1+p_2)\right\}^{1/2}.
 \label{eq:CS-Hellinger}
\end{align}

Recall the identity $H_j(Z+a) = \sum_{r=0}^j \binom{j}{r} a^r H_{j-r}(Z)$ (see \cite[Eq.~(58)]{WeissteinHermitePolynomial}). 
For $Z\sim N(0,1)$ and $j\ge1$, orthogonality yields
\begin{align*}
\Expect H_j(Z+a)^2
 &=\sum_{r=0}^j\binom jr^2a^{2r}(j-r)! \le j!\sum_{r=0}^j\frac{(j a^2)^r}{(r!)^2} 
 \le j!\pth{\sum_{r=0}^\infty \frac{(ja^2)^{r/2}}{r!}}^2 = j!e^{2|a|\sqrt j},
\end{align*}
where the first inequality uses $j!/(j-r)!\le j^r$. The same bound holds for $j=0$, since $H_0=1$.  Applying it coordinatewise and using independence, we obtain
\[
 \Expect_{N(\theta,I)}H_\alpha(X)^2
 \le\alpha!\exp\left\{2\sum_{i=1}^d|\theta_i|\sqrt{\alpha_i}\right\}
 \le\alpha!e^{2\gamma_K\sqrt{|\alpha|}},
 \qquad \theta\in K,
\]
where the second inequality follows from Cauchy--Schwarz and $\norm{\theta}\le \gamma_K$.  Since $\Expect_{N(\theta,I)}H_\alpha(X)=\theta^\alpha$, we have
\[
 \Delta_\alpha(G)=\int H_\alpha(x)\{p_G(x)-p_\star(x)\}\d x.
\]
Applying \prettyref{eq:CS-Hellinger} with $f=H_\alpha$, $p_1=p_G$, and $p_2=p_\star$, the second moments under both mixture are bounded by the supremum of the second moment under $N(\theta, I)$ over $\theta \in K$, so
\[
 |\Delta_\alpha(G)|
 \le 2H\sup_{\theta\in K}\{\Expect_{N(\theta,I)}H_\alpha(X)^2\}^{1/2}
 \le 2H\sqrt{\alpha!}\,e^{\gamma_K\sqrt{|\alpha|}}.
\]
This proves \prettyref{eq:low-coeff} without an atomicity assumption.

Now suppose that $G_\star\in\cP_s(K)$.  For $j\ge1$, let
\[
T_j := \int_K\theta^{\otimes j}\d (G-G_\star)(\theta),
\qquad
\norm{T_j}_{\mathrm{op}}
 := \sup_{\norm{v_1},\ldots,\norm{v_j}\le1}
       |\langle T_j,v_1\otimes\cdots\otimes v_j\rangle|,
\]
and $\pnorm{T_j}{F}$ denotes the usual Frobenius norm of $T_j$. The entries of $T_j$ are the moment differences $\Delta_\alpha(G)$ with $|\alpha|=j$.  Hence 
\[
 \max_{1\le j\le2s}\norm{T_j}_{\mathrm{op}}
 \le\max_{1\le j\le2s}\norm{T_j}_F
 \le\max_{1\le j\le2s}d^{j/2}\max_{|\alpha|=j}|\Delta_\alpha(G)|
 \le C_{K,s}H
\]
where the last step applies \prettyref{eq:low-coeff}. 
Applying the moment comparison in \cite[Lemma~6.3]{ZhangVolgushev2026}, we get, for $\ell>2s$,
\[
 \norm{T_\ell}_{\mathrm{op}}
 \le\ell(\gamma_K+1)^{2s\ell}
       \max_{1\le j\le2s}\norm{T_j}_{\mathrm{op}}
 \le C_{K,s}^{\ell+1}H,
\]
after enlarging $C_{K,s}$.  Each $\Delta_\alpha(G)$ with $|\alpha|=\ell$ is an entry of $T_\ell$, so $|\Delta_\alpha(G)|\le\norm{T_\ell}_{\mathrm{op}}$.  This proves \prettyref{eq:mixed-high}.

\subsection{Proof of Lemma~\ref{lem:Gaussian-translation}}
\label{app:polynomial-translations}

If $m=0$ or $z=0$, the bound is an equality.  Suppose $m\ge1$ and $z\ne0$, and expand $P=\sum_{|\alpha|\le m}c_\alpha\psi_\alpha$ in the orthonormal product-Hermite basis.  The Hermite differentiation identities $\partial_j \psi_\alpha = \sqrt{\alpha_j}\psi_{\alpha-e_j}$ and orthogonality give
\[
 \sum_{j=1}^d\norm{\partial_jP}_{L_2(\varphi)}^2
 =\sum_{|\alpha|\le m}|\alpha|\,|c_\alpha|^2
 \le m\norm P_{L_2(\varphi)}^2.
\]
For $z\cdot\nabla=\sum_{j=1}^dz_j\partial_j$, Cauchy--Schwarz yields
\[
 \norm{(z\cdot\nabla)P}_{L_2(\varphi)}^2
 \le \norm{z}^2\sum_{j=1}^d\norm{\partial_jP}_{L_2(\varphi)}^2
 \le \norm{z}^2m\norm P_{L_2(\varphi)}^2.
\]
Since differentiation lowers total degree, iterating the above inequality $k$ times gives
\[
 \norm{(z\cdot\nabla)^kP}_{L_2(\varphi)}
 \le \norm{z}^k\sqrt{\frac{m!}{(m-k)!}}\norm P_{L_2(\varphi)}
 \le (\norm{z}\sqrt m)^k\norm P_{L_2(\varphi)},
 \qquad 0\le k\le m.
\]
Because $P$ is a degree-$m$ polynomial, an exact Taylor expansion gives  $P(x+z)=\sum_{k=0}^m(z\cdot\nabla)^kP(x)/k!$. Applying the above estimate term by term yields
\begin{align*}
\norm{P(\cdot+z)}_{L_2(\varphi)} \le\sum_{k=0}^m\frac{1}{k!}\norm{(z\cdot\nabla)^kP}_{L_2(\varphi)} \le\left\{\sum_{k=0}^m\frac{(\norm{z}\sqrt m)^k}{k!}\right\}\norm P_{L_2(\varphi)}
 \le e^{\norm{z}\sqrt m}\norm P_{L_2(\varphi)}.
\end{align*}

\subsection{Proof of \prettyref{lmm:esseen}}
\label{app:esseen}
Throughout this proof, $C$ denotes a constant depending on $k$ whose value may change from line to line.
Denote the characteristic function of $Y_1$ by $\phi_{Y_1}(u)=\Expect e^{i\langle u,Y_1\rangle}$, $u\in\reals^k$.
Let $Y_1'$ be an
independent copy of $Y_1$.  Then
\begin{align*}
|\phi_{Y_1}(u)|^2 &=\E\cos\iprod{u}{Y_1-Y_1'}\\
&\leq 1-u^\top\Sigma u + C\Expect|\langle u,Y_1-Y_1'\rangle|^3\\
&\le 1-u^\top\Sigma u+8Cm_3\norm{u}^3
\end{align*}
where the first step is because $|\phi_{Y_1}(u)|^2=\phi_{Y_1}(u)\phi_{Y_1}(-u)$ and $Y_1$ is real-valued, the second is by Taylor expanding $\cos$, and finally we apply Cauchy-Schwarz and $\Expect\norm{Y_1-Y_1'}^3\leq \Expect(\norm{Y_1}+\norm{Y_1'})^3\leq 8m_3$.

For $S_n=(Y_1+\ldots+Y_n)/\sqrt{n}$, by independence we have
$\phi_{S_n}(t)=
\phi_{Y_1}(t/\sqrt{n})^n$, $t\in\R^k$.
Using $1+x\le e^x$ and the above bound on $\phi_{Y_1}$, we obtain
\begin{equation}
 |\phi_{S_n}(t)|
 \le\exp\left\{-\frac12t^\top\Sigma t
                 +\frac{Cm_3\norm{t}^3}{\sqrt n}\right\}.
 \label{eq:ac-cf-sum}
\end{equation}

The multivariate Esseen inequality \cite[Lemma 6.1]{esseen1968concentration} states
\begin{equation}\label{eq:ac-esseen}
Q(S_n,\eps) \le C\eps^k \int_{\norm{t} \le 1/\eps}|\phi_{S_n}(t)|\d t.
\end{equation}
For $\|t\|\leq1/\eps$, by the assumption that 
$\sqrt n\,\eps^3\ge m_3$, we have
$ \frac{m_3\norm{t}^3}{\sqrt n}
 \le\frac{m_3}{\sqrt n\,\eps^3}\le 1$. Furthermore,
$
 e^{-t^\top\Sigma t/2}
 \le e^{1/2}e^{-t^\top(\Sigma+\eps^2I)t/2}$. Thus,
substituting \prettyref{eq:ac-cf-sum} into \prettyref{eq:ac-esseen}, we obtain
\begin{align*}
Q(S_n,\eps) &\le C\eps^k\int_{\R^k}
e^{-t^\top(\Sigma+\eps^2I)t/2}\d t = \frac{C(2\pi)^{k/2}\eps^k}{\sqrt{\det(\Sigma+\eps^2I)}}.
\end{align*}

\section*{Acknowledgment}
Y.W.~is indebted to Yury Polyanskiy and Pengkun Yang for many helpful discussions  in 2020 on the stronger self-regularization for constrained NPMLE, using the approach suggested by \cite[Eq.~(40)]{PolyanskiyWu2020}. While the present proof takes a different route, those discussions particularly on anti-concentration are nevertheless useful. Y.W.~is also grateful to Yury Polyanskiy for bringing the AI-generated manuscript in footnote~\ref{fn:support-divergence} to our attention.

\bibliographystyle{alpha}
\bibliography{multidim_NPMLE}

\end{document}